\documentclass{article}
\usepackage[utf8]{inputenc}
\usepackage[english]{babel}
\usepackage{amsmath,amsfonts,amssymb,amsthm}
\usepackage[paperwidth=8.5in,paperheight=10.5 in]{geometry}
\usepackage{graphicx}
\usepackage{color}
\usepackage[bookmarks=false]{hyperref}
\definecolor{myurlcolor}{rgb}{0.8,0,0}
\hypersetup{colorlinks,
linkcolor=myurlcolor,
citecolor=myurlcolor,
urlcolor=myurlcolor}
\usepackage{mathrsfs}
\usepackage{enumerate}
\usepackage{tikz-cd}
\usetikzlibrary{shapes,snakes}
\usepackage{tikz}
\usetikzlibrary{arrows}

\newtheorem{theorem}{Theorem}[section]
\newtheorem{corollary}[theorem]{Corollary}
\newtheorem{lemma}[theorem]{Lemma}
\newtheorem{proposition}[theorem]{Proposition}

\theoremstyle{definition}
\newtheorem{definition}[theorem]{Definition}

        \newcommand{\be}{\begin{equation}}
        \newcommand{\ee}{\end{equation}}
        \newcommand{\ba}{\begin{eqnarray}}
        \newcommand{\ea}{\end{eqnarray}}
        \newcommand{\ban}{\begin{eqnarray*}}
        \newcommand{\ean}{\end{eqnarray*}}
        \newcommand{\barr}{\begin{array}}
        \newcommand{\earr}{\end{array}}

\renewcommand{\to}{\rightarrow}

\newcommand{\lan}{{\langle}}
\newcommand{\ran}{{\rangle}}

\newcommand{\ep}{{\epsilon}}
\newcommand{\al}{{\alpha}}

\newcommand{\varep}{\varepsilon} %

\newcommand{\tr}{{\rm tr}}

\newcommand{\h}{{\mathcal H}}
\newcommand{\R}{{\mathbb R}}
\newcommand{\C}{{\mathbb C}}

\newcommand{\N}{{\mathbb N}}

\newcommand{\Z}{{\mathbb Z}}
\newcommand{\A}{{\mathscr A}}
\newcommand{\calE}{{\mathcal{E}}}

\newcommand{\bh}{{\mathscr{B}(\mathscr{H})}} 
\newcommand{\Hil}{{\mathscr H}} 

\newcommand{\com}{{\mathcal K}} 

\newcommand{\ds}{\mathfrak d} 
\newcommand{\dixdom}{\mathcal{L}^{(1, \infty)} } 
\newcommand{\trw}{\text{Tr}_w}
\newcommand{\calH}{\mathcal{H}}
\newcommand{\lip}{\text{Lip}} 
\newcommand{\calJ}{\mathcal{J}} 
\newcommand{\Sn}{\mathcal{S}^n_{n_0, h}}
\newcommand{\scrE}{\mathscr{E}}
\newcommand{\diam}{\text{diam}}

\tikzset{
    ncbar angle/.initial=90,
    ncbar/.style={
        to path=(\tikztostart)
        -- ($(\tikztostart)!#1!\pgfkeysvalueof{/tikz/ncbar angle}:(\tikztotarget)$)
        -- ($(\tikztotarget)!($(\tikztostart)!#1!\pgfkeysvalueof{/tikz/ncbar angle}:(\tikztotarget)$)!\pgfkeysvalueof{/tikz/ncbar angle}:(\tikztostart)$)
        -- (\tikztotarget)
    },
    ncbar/.default=0.5cm,
}

\tikzset{square left brace/.style={ncbar=0.5cm}}
\tikzset{square right brace/.style={ncbar=-0.5cm}}

\tikzset{round left paren/.style={ncbar=0.5cm,out=120,in=-120}}
\tikzset{round right paren/.style={ncbar=0.5cm,out=60,in=-60}}

\begin{document}
\title{Spectral Triples for the Variants of the Sierpinski Gasket}
\author{
\begin{tabular}{ccc}
Andrea Arauza Rivera \footnote{Department of Mathematics, University of California, Riverside CA, 92521, USA}
\\
\small arauza@math.ucr.edu 
\end{tabular}
}

\maketitle

\begin{abstract}
Fractal geometry is the study of sets which exhibit the same pattern at multiple scales. Developing tools to study these sets is of great interest. One step towards developing some of these tools is recognizing the duality between topological spaces and commutative $C^\ast$-algebras. When one lifts the commutativity axiom, one gets what are called noncommutative spaces and the study of noncommutative geometry. The tools built to study noncommutative spaces can in fact be used to study fractal sets. In what follows we will use the spectral triples of noncommutative geometry to describe various notions from fractal geometry. We focus on the fractal sets known as the harmonic Sierpinski gasket and the stretched Sierpinski gasket, and show that the spectral triples constructed in \cite{cil} and \cite{lapsar} can recover the standard self-affine measure in the case of the harmonic Sierpinski gasket and the Hausdorff dimension, geodesic metric, and Hausdorff measure in the case of the stretched Sierpinski gasket. 
\end{abstract}

\section{Introduction}

It is a tradition in mathematics to take problems in geometry and turn them into problems in algebra. This shift in perspective often brings with it new approaches and various algebraic tools for solving problems. There is a well-known duality between the category of compact Hausdorff topological spaces and the category of commutative unital $C^\ast$-algebras. Given a compact Hausdorff topological space $X$, one can study various topological properties of $X$ by instead studying the algebraic properties of the $C^\ast$-algebra of continuous functions on $X$, written $C(X)$. To recover the topological space $X$ from the $C^\ast$-algebra $C(X)$, one considers the set of continuous nonzero $\ast$-homomorphisms from $C(X)$ to $\C$, called the Gelfand spectrum of $C(X)$. The Gelfand spectrum of $C(X)$ is homeomorphic to $X$ when given the Gelfand-topology. 

In order to step towards noncommutativity, we use a theorem of Gelfand and Naimark which states that any $C^\ast$-algebra is isometrically $\ast$-isomorphic (i.e. isomorphic as $C^\ast$-algebras) to some closed subalgebra of the bounded operators on a Hilbert space. 
Note that there is no mention of commutativity in the Gelfand-Naimark theorem, so one can drop the commutativity requirement on the $C^\ast$-algebra and study what are known as noncommutative spaces. 
This is the starting point for the study of noncommutative geometry, where one leaves behind the point centered view of geometry and opts for a more algebraic perspective. In order to make this shift in perspective fruitful in the study of fractal geometry, one must have a way of translating important fractal geometric ideas into ideas that can be described with algebraic tools. For this we will use a toolkit known as a spectral triple which consists of three objects, a $C^\ast$-algebra $\A$, a Hilbert space $\Hil$ that carries a faithful representation, $\pi$, from $\A$ to the bounded operators on the Hilbert space, and an essentially self-adjoint unbounded operator $D$ on $\Hil$, satisfying certain conditions.

In our examples, the $C^\ast$-algebra $\A$ will be $C(X)$, where $X$ is one of our fractal sets. Once one has a spectral triple, $(\A, \Hil, D)$, one can begin to formulate notions of dimension, distance, and measure. These are essential when studying the geometry of fractal sets. 

\begin{itemize}
\item For a notion of dimension, one studies the trace of the operators $|D|^{-s}$ for $s>0$, sufficiently large, denoted $\tr(|D|^{-s})$. In our examples, $\tr(|D|^{-s})$ will give a Dirichlet series, and calculating the dimension induced by a spectral triple will amount to finding the abscissa of convergence of this series. 
\item A notion of distance will come from the definitions used in the study of metrics on state spaces found in \cite{co2}, \cite{rif}, \cite{rif2}, \cite{rif3}. For example, on the space of probability measures on a compact metric space, $(X, \rho)$, one can define a metric by
$$ \overline{\rho}(\mu, \nu) = \sup\{|\mu(f) -\nu(f)| :\lip_\rho(f)\leq 1\},$$
where $\lip_\rho(f) = \sup\left\{ \frac{|f(x) - f(y)|}{\rho(x, y)}: x\neq y\right\}$; see \cite{rif}. 
In our setting, the commutator, $[D, \pi(a)]$, where $\pi:\A \to \bh$ is the representation from the spectral triple, will act as a ``derivative" for the element $a\in \A$. The norm $\|[D, \pi(a)]\|$ can then act like a Lipschitz seminorm, of sorts. Thus we may define a metric on a space $X$ by
$$d_X(x, y) = \sup\{|f(x)-f(y)|: f\in C(X), \|[D, \pi(f)]\|\leq 1\}.$$
We note that one of the conditions we impose on the operator $D$ is that the set of $a\in \A$ for which $[D, \pi(a)]$ extends to a bounded operator on $\Hil$, be dense in $\A$. Since $[D, \pi(a)]$ is acting like the derivative of the element $a\in \A$, this condition is essentially guaranteeing the existence of a dense set of ``differentiable" elements in $\A$. This is in analogy to the denseness of $C^1$ functions in $C(X)$. 
\item In order to formulate an operator algebraic notion of measure, one needs a Dixmier trace, $\trw(\cdot)$. One can use a Dixmier trace and the operator $D$ from the spectral triple to create a positive linear functional on the $C^\ast$-algebra $C(X)$. This then gives a measure on the space $X$. The subscript $w$ in the notation $\trw(\cdot)$ indicates the dependence of the Dixmier trace on a choice of extended limit, $w: \ell^\infty \to \C$. For our examples, we determine the measure induced by the Dixmier trace and show that the Dixmier trace is independent of the choice of extended limit, $w$.  
\end{itemize}

Alain Connes \cite{co}, \cite{co2}, proved that one can recover the geometry of a compact spin Riemannian manifold using a spectral triple. In addition, motivated by the work of Michel L. Lapidus and Carl Pomerance \cite{lappo} on fractal strings and their spectra, Connes gives in \cite{co} the construction of a spectral triple for Cantor type fractal subsets of $\R$ and shows that one can recover the Hausdorff measure on these sets.  Since then, various constructions of spectral triples have been used to describe the geometric notions of dimension, distance, and measure on fractals; see \cite{cgis}, \cite{ci}, \cite{cil}, \cite{gu2}, \cite{lap2}, \cite{lap1}, \cite{lapsar}. A program for how one can apply the methods from noncommutative geometry to fractal geometry was given in \cite{lap2} and \cite{lap1} by Michel L. Lapidus. In particular, \cite{lap2} and \cite{lap1} gave methods by which one may connect aspects of noncommutative geometry and analysis on fractals. 
One of the purposes of this work is to generate more examples of spaces for which operator algebraic tools can be used to describe the space's geometry. 

\begin{figure}
\centering
\includegraphics[scale=.18]{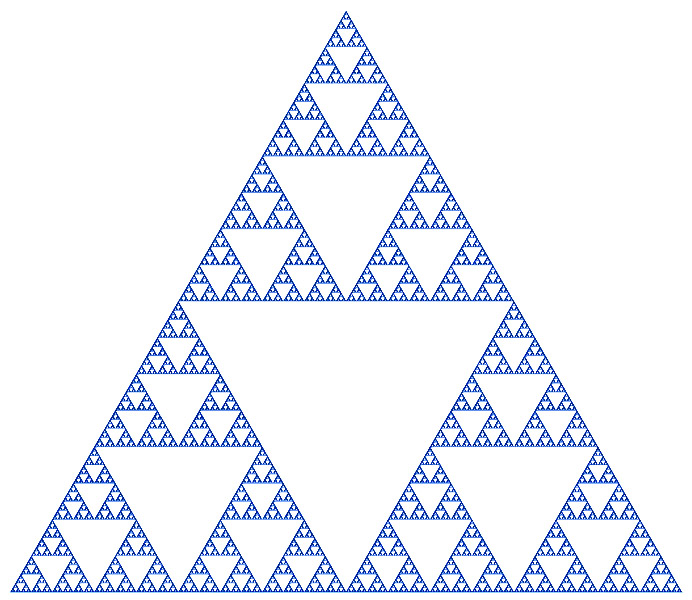} 
\hspace{1in}
\includegraphics[scale = .43]{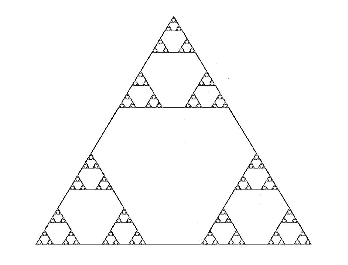}
\caption{Classical Sierpinski gasket (left); Stretched Sierpinski gasket of parameter $\al$ (right) \cite{arf}.}
\label{sierplus}
\end{figure}

The fractal sets that will be the focus of this work are the harmonic Sierpinski gasket, $K_H$, and the stretched Sierpinski gasket, $K_\al$; see Figures \ref{sierplus} and \ref{siertohar1}. These are both variants of (i.e. homeomorphic to) the classical Sierpinski gasket, $SG$, but have features that the classical Sierpinski gasket does not have. For example, the harmonic Sierpinski gasket has the property that between any two points there is a $C^1$ path connecting them, giving a sort of ``smooth fractal manifold" structure; see \cite{kig3}. In order to define the harmonic Sierpinski gasket, we review some basic definitions from the study of analysis on fractals. One of the goals of noncommutative fractal geometry is to establish connections between the use of spectral triples and the study of analysis on fractals. We see that $K_H$ is determined by harmonic functions on $SG$ and that $K_H$ is a self-affine set. The spectral triple we use to study $K_H$ is the spectral triple defined in \cite{lapsar}. One of the open problems stated in \cite{lapsar} concerns the possibility of recovering the Hausdorff measure on $K_H$ by using a spectral triple and another operator algebraic tool, a Dixmier trace. Here we show that this conjecture is false. This shows how working with self-affinity (rather than self-similarity) can lead to complications when attempting to study the fractal geometry of a space.  

The stretched Sierpinski gasket, sometimes also called the Hanoi attractor, is another example of a set which is self-affine and not self-similar. This typically causes complications, especially when trying to find a natural measure with some self-similarity/affinity property. The stretched Sierpinski gasket, $K_\al$, has been the subject of various papers \cite{arf}, \cite{afk}, \cite{artep} which study the sets geometry and develop Dirichlet forms for the space. Once one has a Dirichlet form on $K_\al$, one can study the associated Laplacian and its asymptotics. We give a spectral triple for $K_\al$ and prove various results concerning the recovery of the geometry of $K_\al$. This is a first step towards using noncommutative geometry to study analysis and probability theory on $K_\al$. 

In 2008, Erik Christensen, Cristina Ivan, and Michel L. Lapidus gave the construction of a spectral triple for the Sierpinski gasket and other fractals sets which recovers the Hausdorff dimension, the geodesic distance, and the Hausdorff measure; see \cite{ci}, \cite{cil}. Christensen, Ivan, and Lapidus first built a spectral triple for a circle, then used this to give a spectral triple to each triangle in the graph approximations to $SG$; see Figure \ref{sierstages}. They then defined a spectral triple for $SG$ by taking a direct sum of spectral triples over the triangles in $SG$. In 2015, Michel L. Lapidus and Jonathan J. Sarhad gave the construction of a spectral triple for certain length spaces and showed that their spectral triple recovers the geodesic metric on these length spaces; see \cite{lapsar}. This more general construction of a spectral triple is for length spaces made up of rectifiable $C^1$ curves. Lapidus and Sarhad gave a spectral triple for each rectifiable $C^1$ curve and then took a direct sum to get a spectral triple for the length space. We will give a more detailed description of this construction in Section 3, as this is the spectral triple that we will be using. 

\bigskip

The remaining sections are organized as follows. 
\bigskip

In Section 2, we give the definitions of the classical Sierpinski gasket, the harmonic Sierpinski gasket, and the stretched Sierpinski gasket. We also fix some notation to be used in the remaining sections.

Section 3 includes the definition of a spectral triple and how one may use the tools in a spectral triple to formulate notions of dimension, distance, and measure on sets. We also give a more detailed description of the work in \cite{lapsar}, including the construction of the spectral triple for a rectifiable $C^1$ curve. 

In Section 4, we describe how Lapidus and Sarhad, built a spectral triple for the harmonic Sierpinski gasket. We show that the spectral dimension, $\ds_H$, induced by this spectral triple satisfies
$$1\leq \ds_H \leq \frac{\log{3}}{\log{5} - \log{3}}.$$
In \cite{lapsar}, Lapidus and Sarhad conjectured that the spectral triple they built would recover the Hausdorff measure. We show that this is false and find that the measure recovered by the spectral triple is in fact the unique measure satisfying a certain self-affinity condition. 

Section 5 focuses on the stretched Sierpinski gasket, $K_\al$. We first show that the spectral triple based on the curves that make up $K_\al$ can recover the Hausdorff dimension and the geodesic metric. We note that the stretched Sierpinski gasket does not satisfy the conditions needed for the result on the recovery of the geodesic metric given in \cite{lapsar} and provide a proof of the recovery of the metric here. We then show that the Hausdorff measure on $K_\al$ is the unique measure satisfying a certain self-affinity (really self-similarity) condition and that the measure recovered by the spectral triple for $K_\al$ is just this measure (i.e. the Hausdorff measure). 

We conclude in Section 6 with some remarks on what can be done in the future.


\section{Preliminaries} 
In what follows we will focus on the fractal sets known as the Sierpinski gasket, $SG$, the harmonic Sierpinski gasket, $K_H$, and the stretched Sierpinski gasket, $K_\al$. Let us review how one constructs these sets.

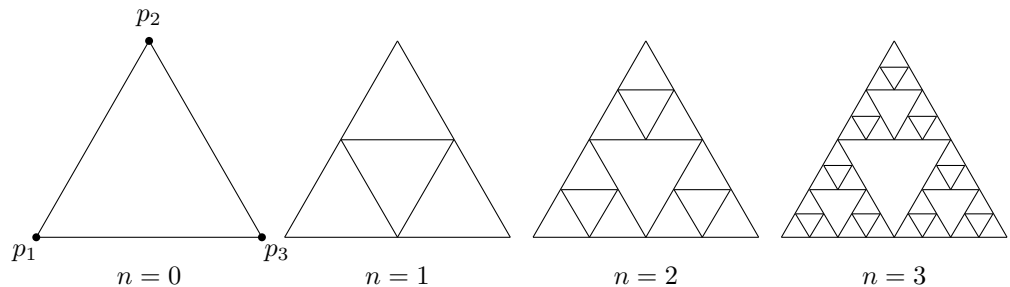
\begin{figure}
\begin{tikzpicture}
\node(a) at (-0.15, -.2) {$p_1$};
\node(a) at (1.5, 2.9) {$p_2$};
\node(a) at (3.2, -.2) {$p_3$};
\node(a) at (1.5, -.5) {$n=0$};


\node[circle,fill=black,inner sep=0pt,minimum size=3pt] (a) at (0,0) {};
\node[circle,fill=black,inner sep=0pt,minimum size=3pt] (a) at (3/2, 2.598) {};
\node[circle,fill=black,inner sep=0pt,minimum size=3pt] (a) at (3,0) {};

\draw[ - ] (0,0) to (3/2, 2.598); 
\draw[ - ] (3/2, 2.598) to (3,0); 
\draw[ - ] (0, 0) to (3, 0); 
\hspace{1.3in}
\node(a) at (1.5, -.5) {$n=1$};

\draw[ - ] (0,0) to (3/2, 2.598); 
\draw[ - ] (3/2, 2.598) to (3,0); 
\draw[ - ] (0, 0) to (3, 0); 
\draw[ - ] (3/4, 1.29) to (1.5, 0); 
\draw[ - ] (1.5, 0) to (9/4,1.29); 
\draw[ - ] (3/4, 1.29) to (9/4, 1.29); 
\hspace{1.3in}
\node(a) at (1.5, -.5) {$n=2$};

\draw[ - ] (0,0) to (3/2, 2.598); 
\draw[ - ] (3/2, 2.598) to (3,0); 
\draw[ - ] (0, 0) to (3, 0); 
\draw[ - ] (3/4, 1.29) to (1.5, 0); 
\draw[ - ] (1.5, 0) to (9/4,1.29); 
\draw[ - ] (3/4, 1.29) to (9/4, 1.29); 
\draw[ - ] (9/8, 1.95) to (12/8, 1.29); 
\draw[ - ] (9/8, 1.95) to (15/8, 1.95); 
\draw[ - ] (15/8, 1.95) to  (12/8, 1.29); 

\draw[ - ] (3/8, .63) to (9/8, .63); 
\draw[ - ] (9/8, .63) to (.75, 0); 
\draw[ - ] (3/8, .63) to (.75, 0); 

\draw[ - ] (15/8, .63) to  (9/4, 0); 
\draw[ - ] ( 21/8, .63) to (9/4,0); 
\draw[ - ] (21/8 , .63) to (15/8, .63); 
\hspace{1.3in}
\node(a) at (1.5, -.5) {$n=3$};

\draw[ - ] (0,0) to (3/2, 2.598); 
\draw[ - ] (3/2, 2.598) to (3,0); 
\draw[ - ] (0, 0) to (3, 0); 
\draw[ - ] (3/4, 1.29) to (1.5, 0); 
\draw[ - ] (1.5, 0) to (9/4,1.29); 
\draw[ - ] (3/4, 1.29) to (9/4, 1.29); 
\draw[ - ] (9/8, 1.95) to (12/8, 1.29); 
\draw[ - ] (9/8, 1.95) to (15/8, 1.95); 
\draw[ - ] (15/8, 1.95) to  (12/8, 1.29); 

\draw[ - ] (3/8, .63) to (9/8, .63); 
\draw[ - ] (9/8, .63) to (.75, 0); 
\draw[ - ] (3/8, .63) to (.75, 0); 

\draw[ - ] (15/8, .63) to  (9/4, 0); 
\draw[ - ] ( 21/8, .63) to (9/4,0); 
\draw[ - ] (21/8 , .63) to (15/8, .63); 
\draw[ - ] (3/8, 0) to (3/16, .31); 
\draw[ - ] (3/8, 0) to (9/16, .31); 
\draw[ - ] (9/16,.31) to  (3/16, .31); 

\draw[ - ] (15/16, .31) to (9/8, 0); 
\draw[ - ] (9/8, 0) to (21/16, .31); 
\draw[ - ] (15/16, .31) to (21/16, .31); 

\draw[ - ] (9/16, .95) to  (3/4, .63); 
\draw[ - ] ( 15/16, .95) to (3/4,.63); 
\draw[ - ] (9/16 , .95) to (15/16, .95); 
\draw[ - ] (15/8, 0) to (27/16, .31); 
\draw[ - ] (15/8, 0) to (33/16, .31); 
\draw[ - ] (27/16,.31) to  (33/16, .31); 

\draw[ - ] (39/16, .31) to (21/8, 0); 
\draw[ - ] (21/8, 0) to (45/16, .31); 
\draw[ - ] (39/16, .31) to (45/16, .31); 

\draw[ - ] (33/16, .95) to  (9/4, .63); 
\draw[ - ] ( 39/16, .95) to (9/4,.63); 
\draw[ - ] (33/16 , .95) to (39/16, .95); 
\draw[ - ] (15/16, 1.6) to (9/8, 1.29); 
\draw[ - ] (21/16, 1.6) to (9/8, 1.29); 
\draw[ - ] (15/16, 1.6) to  (21/16, 1.6); 

\draw[ - ] (27/16, 1.6) to (15/8, 1.29); 
\draw[ - ] (33/16, 1.6) to (15/8, 1.29); 
\draw[ - ] (33/16, 1.6) to (27/16, 1.6); 

\draw[ - ] (21/16, 2.25) to  (3/2, 1.95); 
\draw[ - ] ( 27/16, 2.25) to (3/2, 1.95); 
\draw[ - ] (27/16 , 2.25) to (21/16, 2.25); 
\end{tikzpicture}
\label{sierstages}
\caption{Graph approximations, $SG_n$, of the Sierpinski gasket.}
\end{figure}

The classical Sierpinski gasket seen in Figure \ref{sierplus} can be defined as the closure of an increasing union of graphs. Consider the equilateral triangle, $T$, with vertices
$$ p_1:= (0,0), \ \ \ \ \ \ p_2 := \left(\frac{1}{2}, \frac{\sqrt{3}}{2}\right), \ \ \ \ \ \ p_3 := (1,0).$$
Define the contraction maps 
$$\left\{f_j:\R^2 \to \R^2: f_j(x) := \frac{1}{2}(x-p_j) +p_j, \ \ j=1, 2, 3\right\}.$$ 
Applying these maps to the equilateral triangle $T$ with vertices $p_1, p_2, p_3$ we get an increasing sequence of graphs as in Figure \ref{sierstages}. This gives us graph approximations of the Sierpinski gasket:
$$SG_n := \bigcup_{w\in\{1, 2, 3\}^n} f_w(T) \ \ \ \ \ \ \ \text{ for } n\geq 0$$ 
where $w \in\{1, 2, 3\}^n$ means that $w =w_1w_2\cdots w_n$ is a word in the letters $\{1, 2, 3\}$ with length $|w| = n\geq1$, and $f_w = f_{w_n}\circ \cdots \circ f_{w_2}\circ f_{w_1} $. If $|w| = 0$, then $w = \emptyset$ and $f_\emptyset = id$. 

Notice that one can identify $SG_n$ as a subgraph of $SG_{n+1}$ and hence $\{SG_n\}_{n\geq0}$ is an increasing sequence of graphs. We can define the \textbf{classical Sierpinski gasket} as the closure of the union of these graph approximations:
$$SG:= \overline{\bigcup_{n\geq0} SG_n}.$$
Equivalently, one can define the Sierpinski gasket as the unique non-empty compact set in $\R^2$ which satisfies the self-similarity condition
$$SG = \bigcup_{j=1}^3 f_j(SG).$$ 
Both of these view points will be useful to us in what follows.

Next  let us define the harmonic Sierpinski gasket. This space can be obtained from the classical Sierpinski gasket via a homeomorphism defined using what are known as harmonic functions on $SG$.  For more on the theory of harmonic functions on $SG$ and on analysis on fractals in general see \cite{kig} and \cite{stri}.

%

Let $V_0 = \{p_1, p_2, p_3\}$ and for $n\geq1$ define 
$$V_n = \bigcup_{w\in\{1, 2, 3\}^n} f_w(V_0).$$ 
These are the vertices in the level $n$ approximation to the Sierpinski gasket, $SG_n$. Let 
$$V^* = \bigcup_{n\geq0}V_n.$$
 
 \begin{figure}
\[ \includegraphics[scale= .4]{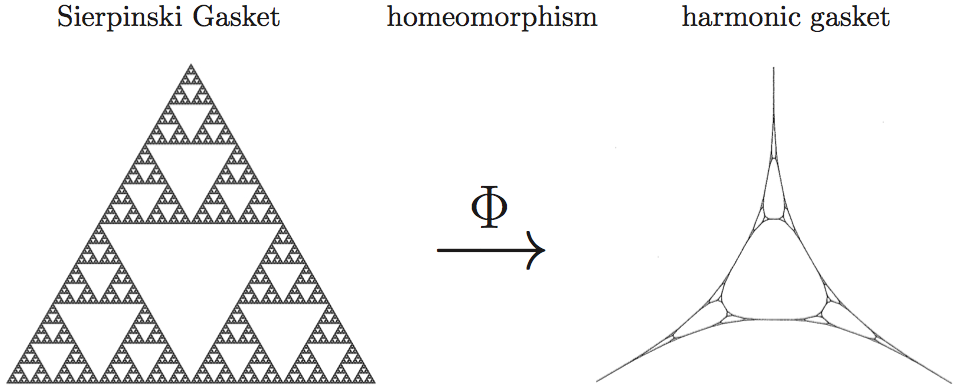}\]
\caption{\cite{lapsar}}
\label{siertohar1}
\end{figure}

\begin{definition}
Given $f, g : V_n \to \R$ define the \textbf{energy} on $SG_n$ by
$$E_n(f, g) :=  \sum_{x\sim_n y} (f(x)-f(y))(g(x)-g(y)),$$  
where $x\sim_ny$ means $x, y\in F_{w}(V_0)$ for some $w\in \{1, 2, 3\}^n$. That is, $x$ and $y$ are connected by an $n$-edge in $SG_n$. In the sum, we count each pair $x, y$ with $x\sim_ny$ exactly once.
 
We will focus on the case when $f=g$ so
$$E_n(f) := E_n(f, f) =  \sum_{x\sim_n y} (f(x)-f(y))^2.$$  
\end{definition}

Given $f: V_n \to \R$ we can extend $f$ to $V_{n+1}$ in many ways. If we extend so that $E_{n+1}(f)$ is as small as possible, the extension is called the harmonic extension of $f$ to $V_{n+1}$.

\begin{definition}
A function $f:V_n \to \R$ is \textbf{harmonic} if given its values at $V_0$ it minimizes $E_k(f)$ for each $k =1, 2, \dots, n$.
\end{definition}

A calculation shows that for a harmonic function $f:V_{n+1}\to\R$ we have
$$E_{n}(f) = \frac{5}{3} E_{n+1}(f).$$ 

\begin{definition}
Given $f: V_n \to \R$ define the \textbf{renormalized energy} on $SG_n$ by
$$\calE_0(f) := E_0(f) \ \ \ \ \text{ and } \ \ \ \ \calE_{n}(f) := \left(\frac{5}{3}\right)^{n} E_{n}(f) \text{ for } n\geq1.$$
\end{definition}

So long as $f$ is extended harmonically, the quantity $\calE_n(f)$ is constant as $n$ increases. Otherwise, $\calE_n(f)$ increases as $n$ increases. This means the limit 
$$\calE(f) := \lim_{n\to\infty}  \calE_n(f)$$
exists (but is possibly infinite) for $f:V^* \to \R$.

\begin{figure}
\[ \large
\begin{tikzcd}
SG \arrow{r}{f_j} \arrow{d}{\Phi} & SG \arrow{d}{\Phi} \\
K_H  \arrow{r}{H_j} & K_H
\end{tikzcd} \]
\caption{\cite{kig1}}
\label{comdi}
\end{figure}

The set $V^*$ is dense in $SG$ and hence a uniformly continuous function on $V^*$ can be uniquely extended to a function on all of $SG$. One can show that harmonic functions on $V^*$, and in fact functions for which the limit in $\calE(f)$ is finite, are uniformly continuous on $V^*$; see \cite{kig}, \cite{stri}. 
For each $j=1, 2, 3$, consider the function $h_j: SG \to \R$, where $h_j(p_k) = \delta_{j}(k)$ for $k=1, 2, 3$ and $h_j$ is extended harmonically to $V^*$ and by continuity to the Sierpinski gasket, $SG$. 

Define $\Phi: SG \to \R^3$ by
$$\Phi(x) = \frac{1}{\sqrt{2}} \left(\left (\begin{array}{c} h_1(x)\\ h_2(x)\\ h_3(x) \end{array}\right)-
\frac{1}{3} \left( \begin{array}{c} 1\\1\\1 \end{array} \right)\right).$$
We define the \textbf{harmonic Sierpinski gasket} by $K_H:= \Phi(SG)$. It was shown by Kigami in \cite{kig1} that $\Phi$ is a homeomorphism between $SG$ and $K_H$ when endowing these spaces with the topology induced by the restriction of the Euclidean metric. 

We can also define $K_H$ in terms of contraction maps, as was done for the classical Sierpinski gasket. 
Let $Z = \{(x,y,z) \in\R^3: x+y+z = 0\}$ and let 
$$P = \frac{1}{3}\left( \begin{array}{ccc}
2 & -1 & -1 \\
-1 & 2 & -1 \\
-1 & -1 & 2 \end{array} \right)$$ 
be the orthogonal projection of $\R^3$ onto $Z$. Let $q_j = \frac{3P(b_j)}{\sqrt{6}}$ for $j =1, 2, 3 $ where $\{b_1, b_2, b_3\}$ is the standard basis for $\R^3$. Choose $q'_j \in \R^3$ such that $\{q_j, q'_j\}$ is an orthonormal basis for $Z$. For $j = 1, 2, 3,$ define $M_j:Z \to Z$ by 
$$M_j(q_j)=\frac{3}{5}q_j \ \ \text{ and } \ \ M_j(q'_j) = \frac{1}{5} q'_j$$ 
and let $H_j: Z \to Z$ be given by 
\begin{equation}\label{hmaps}
H_j(x) = M_j(x-q_j)+q_j.
\end{equation} 
The maps $H_j$ are contractive affine maps and $K_H$ is the unique non-empty compact set such that 
$$K_H = \bigcup_{n=1}^3 H_j(K_H).$$

The two equivalent ways of defining the harmonic Sierpinski gasket are connected via the relation 
$\Phi\circ f_j = H_j \circ \Phi$ (for $j=1, 2, 3$) or the commutative square in Figure \ref{comdi}; see \cite{kig1}.

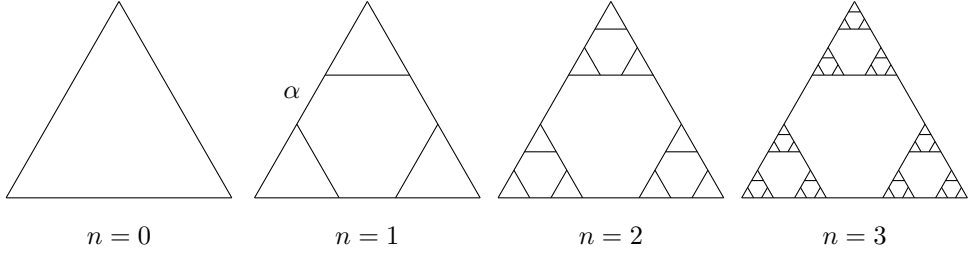
\begin{figure} 
\begin{tikzpicture}
\node(a) at (1.5, -.5) {$n=0$};
\draw[ - ] (0,0) to (3/2, 2.598); 
\draw[ - ] (3/2, 2.598) to (3,0); 
\draw[ - ] (0, 0) to (3, 0); 
\hspace{1.3in}
\node(d) at (.5, 1.4) {$\alpha$};
\node(a) at (1.5, -.5) {$n=1$};

\draw[ - ] (0,0) to (3/2, 2.598); 
\draw[ - ] (3/2, 2.598) to (3,0); 
\draw[ - ] (0, 0) to (3, 0); 
\draw[ - ] (.5625, .975) to (1.125, 0); 
\draw[ - ] (2.4375, .975) to (1.875,0); 
\draw[ - ] (.9375, 1.6234) to (2.0625, 1.6234); 
\hspace{-.3 in}
\node(a) at (5.5, -.5) {$n=2$};
\draw[ - ] (4,0) to (3/2+4, 2.598); 
\draw[ - ] (3/2+4, 2.598) to (7,0); 
\draw[ - ] (4, 0) to (7, 0); 
\draw[ - ] (4.5625, .975) to (5.125, 0); 
\draw[ - ] (6.4375, .975) to (5.875,0); 
\draw[ - ] (4.9375, 1.6234) to (6.0625, 1.6234); 
\draw[ - ] (4.211, .3655) to (4.421875, 0); 
\draw[ - ] (4.35155, .609) to (4.7735, .609); 
\draw[ - ] (4.703, 0) to (4.914, .3655); 

\draw[ - ] (5.1485, 1.9895) to (5.3595, 1.6238); 
\draw[ - ] (5.289, 2.2325) to (5.7109, 2.2325); 
\draw[ - ] (5.641, 1.6238) to (5.8515, 1.9895); 

\draw[ - ] (6.0859, .3655) to (6.2969, 0); 
\draw[ - ] (6.2265, .6089) to (6.6485, .6089); 
\draw[ - ] (6.789, .3655) to (6.578, 0); 

\hspace{2.85in}
\node(a) at (1.5, -.5) {$n=3$};

\draw[ - ] (0,0) to (3/2, 2.598); 
\draw[ - ] (3/2, 2.598) to (3,0); 
\draw[ - ] (0, 0) to (3, 0); 
\draw[ - ] (0.5625, .975) to (1.125, 0); 
\draw[ - ] (2.4375, .975) to (1.875,0); 
\draw[ - ] (0.9375, 1.6234) to (2.0625, 1.6234); 
\draw[ - ] (0.211, .3655) to (0.421875, 0); 
\draw[ - ] (0.35155, .609) to (0.7735, .609); 
\draw[ - ] (0.703, 0) to (0.914, .3655); 

\draw[ - ] (1.1485, 1.9895) to (1.3595, 1.6238); 
\draw[ - ] (1.289, 2.2325) to (1.7109, 2.2325); 
\draw[ - ] (1.641, 1.6238) to (1.8515, 1.9895); 

\draw[ - ] (2.0859, .3655) to (2.2969, 0); 
\draw[ - ] (2.2265, .6089) to (2.6485, .6089); 
\draw[ - ] (2.789, .3655) to (2.578, 0); 
\draw[ - ] (0.0791, .1370) to (0.158, 0); 
\draw[ - ] (0.1318, .2283) to (0.290, .2283); 
\draw[ - ] (0.264, 0) to (0.343, .137); 

\draw[ - ] (1.368, 2.37) to (1.446,2.232); 
\draw[ - ] (1.421, 2.46) to (1.579,2.46); 
\draw[ - ] (1.553, 2.233) to (1.632, 2.37); 

\draw[ - ] (2.657, 0.137) to (2.738,0); 
\draw[ - ] (2.710, .228) to (2.868,.228); 
\draw[ - ] (2.921, .137) to (2.842, 0); 

\draw[ - ] (1.954, 0.137) to (2.033,0); 
\draw[ - ] (2.007, .228) to (2.165,.228); 
\draw[ - ] (2.139, 0) to (2.212, .137); 

\draw[ - ] (.431, 0.746) to (.509,0.609); 
\draw[ - ] (.4834, .837) to (.642,.837); 
\draw[ - ] (.615, 0.609) to (.6945, .746); 

\draw[ - ] (.782, .137) to (.863, 0); 
\draw[ - ] (.835, .228) to (.993, .228); 
\draw[ - ] (1.046, .137) to (.9668, 0); 

\draw[ - ] (1.017, 1.761) to (1.096, 1.624); 
\draw[ - ] (1.07, 1.852) to (1.223, 1.852); 
\draw[ - ] (1.201, 1.624) to (1.28, 1.761); 

\draw[ - ] (2.306, .746) to (2.385, .609); 
\draw[ - ] (2.358, .837) to (2.517, .837); 
\draw[ - ] (2.490, .609) to (2.570, .746); 

\draw[ - ] (1.72, 1.761) to (1.8, 1.624); 
\draw[ - ] (1.773, 1.852) to (1.931, 1.852); 
\draw[ - ] (1.983, 1.761) to (1.904, 1.624); 
\end{tikzpicture} 
\label{ssgstages}
\caption{Graph approximations of the stretched Sierpinski gasket.}
\end{figure}

Next we define the stretched Sierpinski gasket. 
Fix $\al \in (0, \frac{1}{3})$ and let $p_1, p_2, \dots, p_6 \in \R^2$ be given by
\[
\begin{tabular}{ l  c  r}
$p_1 = (0,0)$, & $p_2 = \left(\frac{1}{2}, \frac{\sqrt{3}}{2}\right)$, & $p_3 = (1, 0)$,\\

$p_4 = \frac{p_2+p_3}{2},$ & $p_5 = \frac{p_1+p_3}{2}$, & $p_6 = \frac{p_1+p_2}{2}$.\\
\end{tabular}
\]
Let $A_1, A_2, \dots, A_6$ be $2\times 2$ matrices given by 
$$A_1=A_2= A_3 = \frac{1-\al}{2}\left( \begin{array}{cc}  
 1& 0  \\
0 & 1 \end{array} \right), $$
$$ A_4 = \frac{\al}{4}\left( \begin{array}{cc} 
 1& -\sqrt{3} \\
-\sqrt{3} & 3 \end{array} \right),\ \
A_5 =\al \left( \begin{array}{cc} 
 1& 0  \\
0 & 0 \end{array} \right),\ \  
A_6= \frac{\al}{4}\left( \begin{array}{cc} 
 1& \sqrt{3}  \\
\sqrt{3} & 3 \end{array} \right). $$

Define the maps $F_{\al, j} :\R^2 \to\R^2$ by 
\begin{equation}\label{defssg}
F_{\al, j}(x):= A_j(x-p_j)+p_j \text{ for } j=1, 2, \dots, 6.
\end{equation} 
The maps $F_{\al,1}, F_{\al, 2}, F_{\al, 3}$ will map the triangle $T$ to smaller triangles at each of the three corners of $T$. Note that $F_{\al,1}, F_{\al, 2}, F_{\al, 3}$ are contraction similarities, meaning that they are maps which shrink the space by the same ration, namely $\frac{1-\al}{2}$, in every directing. On the other hand, the maps $F_{\al, 4}, F_{\al, 5}, F_{\al, 6}$ will map $T$ to line segments of length $\al$ and these are contractive affine maps, meaning that they shrink the space but may do so by different ratios in different directions. As with the classical Sierpinski gasket, we can define $K_\al$ as the closure of the increasing union of the graphs seen in Figure \ref{ssgstages} or as the unique set satisfying some condition involving the maps $\{F_{\al, j}\}$. We define the \textbf{stretched Sierpinski gasket} as the unique non-empty compact set $K_\al \subseteq \R^2$ satisfying the self-affinity condition
$$K_\al = \bigcup_{j=1}^6 F_{\al, j}(K_\al).$$


\begin{figure}
\centering
\begin{tikzpicture}
\node(a) at (5/8, 2) {$ e^3$};
\node(a) at (31/8, 2) {$e^1$};
\node(a) at (9/4, -.5) {$e^2$};

\node(a) at (-0.3, 0) {$p_1$};
\node(a) at (9/4, 4.17) {$p_2$};
\node(a) at (4.87, 0) {$p_3$};
\node(a) at  (1.5, 1.9) {$p_6$};
\node(a) at  (3.1, 1.9) {$p_4$};
\node(a) at  (9/4, 0.3) {$p_5$};
\node[circle,fill=black,inner sep=0pt,minimum size=3pt] (a) at (0,0) {};
\node[circle,fill=black,inner sep=0pt,minimum size=3pt] (a) at (9/4,3.88) {};
\node[circle,fill=black,inner sep=0pt,minimum size=3pt] (a) at (9/2,0) {};
\node[circle,fill=black,inner sep=0pt,minimum size=3pt] (a) at (9/8,1.92) {};
\node[circle,fill=black,inner sep=0pt,minimum size=3pt] (a) at (27/8,1.92) {};
\node[circle,fill=black,inner sep=0pt,minimum size=3pt] (a) at (9/4,0) {};

\draw[ - ] (0,0) to (9/4, 3.88); 
\draw[ - ] (9/4, 3.88) to (9/2,0); 
\draw[ - ] (0, 0) to (9/2, 0); 
\draw[ - ] (.843, 1.46) to (1.68, 0); 
\draw[ - ] (3.64, 1.46) to (2.81,0); 
\draw[ - ] (1.405, 2.43) to (3.09, 2.43); 

\draw[ (-) ] (.843, 1.46) to (1.405, 2.43);
\draw[(-)] (3.64, 1.46) to (3.09, 2.43); 
\draw[ (-) ] (1.68, 0) to (2.81,0); 
\end{tikzpicture}
\caption{The triangle $T$ with fixed points $p_j$ of the maps $F_j$ and edges $e^1, e^2, e^3$.}
\label{ptsandedges}
\end{figure}
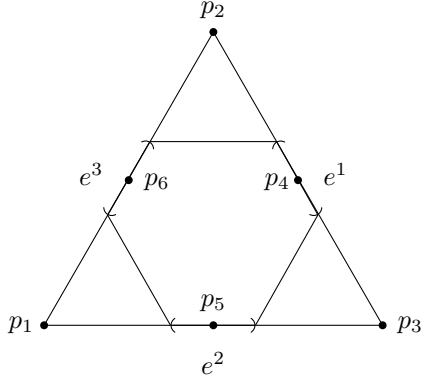

We will fix the parameter $\al\in (0, \frac{1}{3})$ and hence will write $F_{\al, j} = F_j$. If $\al = 0$, then $K_\al = SG$, and if $\al = 1/3$, then the geometry of the space reduces to the 1-dimensional case. 

Notice that $K_\al$ can be written in terms of a ``discrete" part and a ``continuous" part. Let $W_\al$ be the unique compact set satisfying 
$$W_\al := \bigcup_{j=1}^3 F_j(W_\al).$$
Let $J_0 = \emptyset$ and for $n\geq1$ let
$$J_{\al, n} = J_n := \bigcup_{m=0}^{n-1} \ \ \bigcup_{w\in\{1, 2, 3\}^m} F_w\left( \bigcup_{j=1}^3 e^j\right)$$
where $e^j = \text{int}(F_{j+3}(T))$ for $j =1, 2, 3$. Note that $e^1, e^2, e^3$ are the three edges in the first graph approximation of $K_\al$ which join the three triangles in $K_\al$ together. We will call the edges in $J_n$, the level $n$ joining edges.  Also make note of the fact that we take the level $n$ joining edges to be open. Letting $J^* = \cup_{n\geq1} J_n$ we see that
$$K_\al = \bigcup_{j=1}^6 F_j(K_\al) = W_\al \ \dot \cup \ J^*$$
where the second union is disjoint; see \cite{afk}. The set $W_\al$ is the \emph{discrete} part of $K_\al$ and has many properties similar to the classical Sierpinski gasket; the set $J^*$ is the \emph{continuous} part of $K_\al$ and is a union of shrinking intervals. This decomposition of $K_\al$ will be essential in proving results concerning the Hausdorff dimension and measure of $K_\al$. 

\section{Spectral triples}
We will now introduce the toolkit from noncommutative geometry that we use to study fractal sets like the stretched Sierpinski gasket and the harmonic Sierpinski gasket. We use the notation $[A, B]:= AB - BA$ for the commutator of two operators $A, B$ on a Hilbert space. Also, given a Hilbert space $\Hil$, we write $\bh$ for the space of bounded operators on $\Hil$. 

\begin{definition}
A \textbf{spectral triple} $(\A, \Hil, D)$ is a collection of three objects
\begin{itemize}
\item $\A$ a unital $C^\ast$-algebra, 
\item $\Hil$ a Hilbert space which carries a unital faithful representation $\pi:\A \to \bh$, and 
\item an unbounded, essentially self-adjoint, operator $D$ with domain, $\text{Dom}(D) \subseteq \Hil$, such that
	\begin{enumerate}[(a)]
	\item \label{cond1onD} the set 
	$$\{a\in \A : [D, \pi(a)] \text{ is densely defined and has a bounded extension to }\Hil\},$$
	is dense in $\A$ and 
	\item \label{cond2onD} the operator $(I+D^2)^{-1}$ is compact. 
	\end{enumerate}
\end{itemize}
\end{definition}

The $C^\ast$-algebra $\A$ will often be $C(X)$, where $X$ is a compact Hausdorff space. The operator $[D, \pi(a)]$, for $a\in \A$, will act like the ``derivative" of the element $a$ and the dense set in condition (\ref{cond1onD}) will act like the set of $C^1$ functions in $C(X)$.

\subsection{Induced notions of dimension, distance, and measure} 
Using the three tools in a spectral triple, one can define notions of dimension, metric, and measure on a compact Hausdorff space $X$. 

\begin{definition}
Given a spectral triple $(C(X), \Hil, D)$, the number 
$$\ds = \ds(X):= \inf\{p>0: \tr((I+D^2)^{-p/2})<\infty\}$$ 
is the \textbf{spectral dimension} (or \textbf{metric dimension}) of the space $X$.
\end{definition}

Note that condition $(\ref{cond2onD})$ in the definition of a spectral triple is needed so that the trace in the definition of spectral dimension has a possibility of being finite. A priori there is no reason why the spectral dimension $\ds$ should be finite. 

We next define a notion of distance induced by a spectral triple. The definition will look familiar to those who know of the metrics on state spaces. For more on this, see the works of Marc Rieffel in \cite{rif}, \cite{rif2}, \cite{rif3} and of Alain Connes in \cite{co2}.

\begin{definition}
Given a spectral triple $(C(X), \Hil, D)$, define the \textbf{spectral distance} by \\
$$d_X(x, y) = \sup\{|f(x)-f(y)|: f\in C(X), \|[D, \pi(f)]\|\leq 1\},$$
for $x, y\in X$.
\end{definition}
 \bigskip
 
Using a spectral triple and another notion from noncommutative geometry we can define a notion of measure. For this we must introduce the concept of the Dixmier trace.
 
The space in the definition that follows is an ideal in the set of compact operators and will serve as the domain of the Dixmier trace. 
For a compact operator $T$, denote by $\mu_j(T)$ the eigenvalues of $|T|$ ordered so that $0\leq \mu_{j+1}(T) \leq \mu_{j}(T) \text{ for } j \in \N$. 

\begin{definition}
Let $\Hil$ be a separable Hilbert space. Define 
$$\dixdom = \left\{T\in \com : \|T\|_{(1, \infty)}:=\sup_N \frac{1}{\log{(1+N)}} \sum_{j=1}^N \mu_j(T)<\infty\right\},$$
 where $\com$ is the set of compact operators on $\Hil$.
\end{definition}

Let $w:\ell^\infty \to \C$ be a linear functional which vanishes on $c_0$ and satisfies, $w(x_1, x_1, x_2, x_2, \dots) = w(x_1, x_2, \dots)$ for $(x_n) \in \ell^\infty$. This dilation invariance is a technical requirement to ensure that the Dixmier trace is linear on positive operators in $\dixdom$. The existence of such linear functionals is given by the group action invariant Hahn-Banach theorem stated in Theorem 3.3.1 in \cite{ed}. 

\begin{definition}
The \textbf{Dixmier trace} of $T\in\dixdom$, where $T\geq0$, is given by
$$\text{Tr}_w(T) = w\left\{\frac{1}{\log{(1+N)}} \sum_1^N \mu_n(T)\right\}.$$ 
Since any self-adjoint operator is the difference of two positive operators and any bounded operator is the linear combination of self-adjoint operators, we can define $\text{Tr}_w(\cdot)$ for an arbitrary compact operator in $\dixdom$ by linearity.
\end{definition}

Note the dependence of $\trw$ on the choice of $w$. The results that follow will show cases in which $\trw$ is independent of the choice of $w$, but this is not true in general. See \cite{lord} for more on the theory of singular traces such as the Dixmier trace. 

The following theorem of Alain Connes in \cite{co} is often used to compute the Dixmier trace as the residue of a certain series. We will make use of this theorem in the sections that follow.

\begin{theorem}\label{co}\ \\
For $T\geq 0$, $T \in \dixdom$, the following two conditions are equivalent:
\begin{enumerate}
\item $(s-1)\sum_{n=0}^\infty \mu_n(T)^s \to L$ as $s\to 1^+$;
\item $\frac{1}{\log N} \sum_{n=0}^{N-1} \mu_n(T) \to L$ as $N\to \infty$.
\end{enumerate}
\end{theorem}
\medskip

A result of Connes is that for a suitable choice of spectral triple, the map $\trw(\pi(f) |D|^{-\ds})$ is a non-trivial positive linear functional on $C(X)$ and hence induces a measure; see \cite{co}. This is how we will use a spectral triple to induce a measure on a fractal set.

\subsection{Spectral triple for a curve}

One can now associate to a spectral triple $(C(X), \Hil, D)$ a notion of dimension, metric, and measure. The construction of a spectral triple that we use for the classical Sierpinski gasket, the stretched Sierpinski gasket, and the harmonic Sierpinski gasket is based on the construction of a spectral triple for a continuous curve. The following construction of a spectral triple for a curve was first examined by Christensen, Ivan, and Lapidus in \cite{cil} and later used by Lapidus and Sarhad in \cite{lapsar}. 
We define a spectral triple for a curve as follows.

\begin{definition}
Let $X$ be a compact Hausdorff space, $\ell >0$, and $R:[0, \ell] \to X$ a continuous injective map. Then a spectral triple for the $R$-curve is
\begin{itemize}
\item $C(X)$;
\item $\Hil_\ell:= L^2\left([-\ell, \ell] , (2\ell)^{-1}m\right)$ where $(2\ell)^{-1}m$ is the normalized Lebesgue measure and the representation is given by $\pi_\ell : C(X)\to \mathscr{B}(\Hil_\ell)$, 
$$\pi_\ell(f)h(x) := f(R(|x|))h(x);$$
\item $D_{\ell} := D + \frac{\pi}{2\ell} I$ where $D$ is the closure of the operator $-i\frac{d}{dx}$ restricted to the linear span of the set $\{\phi^\ell_k = e^{i\pi kx/\ell}: k\in\Z\}$. That is, $D = \overline{-i\frac{d}{dx}|_{span(\phi^\ell_k)}}$.
\end{itemize}
\end{definition}

Note that $\{\phi_k^\ell(x) = e^{i\pi kx/\ell}\}_{k\in \Z} $ is an orthonormal basis for $\Hil_\ell$ and that these are eigenfunctions of the operator $-i\frac{d}{dx}$, with eigenvalues $\{\frac{\pi k}{\ell}: k\in\Z\}$. We consider functions in $\Hil_\ell$ as restrictions of $2\ell$-periodic functions on $\R$ and hence the operator $D_\ell$ has periodic boundary conditions. The translation in the definition of the operator $D_\ell$ is needed in order to ensure that 0 is not an eigenvalue of the operator. This allows us to talk about the eigenvalues of the operator $|D_\ell|^{-1}$. 

The eigenvalues of $D_\ell$ are 
$$\sigma(D_\ell) = \left\{\frac{(2k+1)\pi}{2\ell}: k\in \Z\right\}$$
and the operator $D_\ell$ can be defined for $f \in L^2[-\ell, \ell]$ by 
$$ D_\ell f = \sum_{k\in\Z} \frac{(2k+1)\pi}{2\ell} \lan f, \phi^\ell_k\ran \phi^\ell_k,$$
where we say that $f$ is in the domain of $D_\ell$, written Dom$(D_\ell)$, if 
$$\|D_\ell f \|_2^2 = \sum_{k\in\Z} \left|\frac{(2k+1)\pi}{2\ell} \right|^2 |\lan f, \phi^\ell_k\ran|^2 < \infty.$$

We think of functions in $C(X)$ as functions in $L^2([-\ell, \ell])$ by working with $f(R(|x|))\in L^2([-\ell, \ell])$ rather than $f\in C(X)$. In particular, note that we care about the a.e. \emph{equivalence class} of the function $f(R(|x|))$ in $L^2([-\ell, \ell])$. 
It is also important to note that for functions $f(R(|x|))\in C^1([-\ell, \ell])$ and $g \in C^1([-\ell, \ell])\subseteq  L^2([-\ell, \ell])$ we have
$$[D_\ell, \pi_\ell(f)] g = \pi_{\ell}\left(-i \frac{df}{dx}\right) g = \pi_{\ell}(Df) g.$$
This shows that the operator $[D_\ell, \pi_\ell(f)]$ is densely defined 
and extends to the bounded operator $\pi_\ell(D f)$ on $L^2([-\ell, \ell])$. 
Proposition 4.1 in \cite{cil} shows that the set in condition (\ref{cond1onD}) of the definition of a spectral triple, is dense in $C(X)$.  It follows that the above is indeed a spectral triple for the $R$-curve. 

The following lemma was stated in \cite{cil}. 

\begin{lemma}\label{lemfromcil}
Let $f:[ -\ell, \ell] \to \C$ be a continuous function. Then the following are equivalent:
\begin{enumerate}
\item $[D_\ell, \pi_\ell(f)]$ is densely defined and bounded.
\item $f \in \text{Dom}(D)$ and $Df$ is essentially bounded.
\item There exists a measurable, essentially bounded function $g: [-\ell, \ell]\to \C$ such that
$$\int_{-\ell}^\ell g(t)\ dt = 0 \text{ and for all } x\in [-\ell, \ell]: \ \ \ f(x) = f(0)+ \int_0^x g(t) \ dt.$$
\end{enumerate}
If the conditions above are satisfied then $g(x) = (i Df)(x)$ almost everywhere.
\end{lemma}

Using curve spectral triples, Christensen, Ivan, and Lapidus constructed a spectral triple for the classical Sierpinski gasket that recovers the Hausdorff dimension, the geodesic metric, and the $\log_23$-dimensional Hausdorff measure. Later, Lapidus and Sarhad used the spectral triple for an $R$-curve to build a spectral triple for compact length spaces $X\subseteq \R^n$ satisfying the axioms below. We write $L(\gamma)$ for the length of the path $\gamma$ parameterized by arclength.

\begin{enumerate}[\textbf{Axiom} 1.]
\item $X =  \overline{\mathcal{R}}$ where $\mathcal{R} = \bigcup_{j=1}^\infty R_j$ and each $R_j$ is a $C^1$ rectifiable curve such that $L(R_j) \to 0$ as $j\to\infty$.
\item There is a dense set $\mathcal{B} \subset X$ which is such that for each $p\in \mathcal{B}$ and $q\in X$ one of the minimizing geodesics from $p$ to $q$ is given by a countable (or finite) concatenation of the $R_j$'s.  
\end{enumerate}

Notice that these two Axioms imply that $\mathcal{B}$ is a subset of the set of endpoints of the $R_j$. It follows that the set of endpoints of the $R_j$'s is dense in $X$. 
Proposition 1 in \cite{lapsar} states that for a compact length space $X$ satisfying Axiom 1, the direct sum of the spectral triples for the $R_j$ curves making up $X$ gives a spectral triple for $X$ and the operator $D$ in that spectral triple has eigenvalues
$$\sigma(D) = \bigcup_{j\geq0}\left\{\frac{(2k+1)\pi}{2\ell_j} : k\in\Z\right\},$$
where $\ell_j := L(R_j)$. 
Furthermore, in Theorem 2 of \cite{lapsar} Lapidus and Sarhad prove that for a compact length space $X$ with Axioms 1 and 2, the spectral distance induced by the direct sum spectral triple and the geodesic distance on $X$ are the same:
$$d_X(x, y) = d_{geo}(x, y) \ \ \text{ for } x, y\in X.$$

This result shows that the direct sum spectral triple for the classical Sierpinski gasket and for the harmonic Sierpinski gasket recovers geodesic distance. If one takes for the curves $R_j$ the edges of the triangles and the joining edges in the stretched Sierpinski gasket, $K_\al$,  then Axiom 2 is not satisfied and hence the theorem of Lapidus and Sarhad does not give that the spectral metric is the same as the geodesic metric on $K_\al$.  We will prove the recovery of the geodesic distance on $K_\al$ in the following section. In addition, we will show that the direct sum spectral triple for $K_\al$ recovers the Hausdorff dimension and Hausdorff measure on $K_\al$.
 
It was conjectured in \cite{lapsar} that the Hausdorff measure on the harmonic Sierpinski gasket $K_H$ with the geodesic metric can be recovered by the direct sum spectral triple via the Dixmier trace. In the section that follows we will show that the Dixmier trace recovers the standard self-affine measure on the harmonic Sierpinski gasket but does not recover the Hausdorff measure on $K_H$. 


\section{A spectral triple for the harmonic Sierpinski gasket, $K_H$.}
First we define curves which correspond to the edges in the graphs which approximate $SG$ and then get curves for the edges in $K_H$ via the homeomorphism $\Phi$. Let $R_j$ for $j\geq1$ be the continuous injective functions which map to the edges in the graphs $SG_n$:
\begin{align*}
& R_j:[0, 1] \to \R^2 \text{ for } j=1, 2, 3 \text{ be the edges in the graph } SG_0, \\
& R_j:[0, 2^{-1}] \to \R^2 \text{ for } j = 4, 5, \dots, 12 \text{ be the edges in the graph } SG_1, \\
& R_j:[0, 2^{-2}] \to \R^2 \text{ for } j = 13, 14, \dots, 39 \text{ be the edges in the graph } SG_2,
\end{align*}
and so on. The curves we use to build spectral triples are parameterized by arc length and the sets $R_j([0, 2^{-k}])$ are precisely the edges in the graph approximations of $SG$. For simplicity we write $R_j$ for $R_j([0, 2^{-k}])$. One can show that 
$$SG = \overline{ \bigcup_{j\geq1} R_j}$$
 (see \cite{lapsar}).

Applying the map $\Phi: SG \to K_H$, we get curves $\Phi(R_j)$. Set $\ell_j =  L(\Phi(R_j))$ and after a reparameterization we have curves
$$\{\Phi(R_j):[0, \ell_j]\to K_H \}_{j=1}^ \infty.$$
Again one can show that
$$K_H = \overline{ \bigcup_{j\geq1} \Phi(R_j)}$$ 
 (see \cite{lapsar}).
The direct sum of the spectral triples for the curves $\Phi(R_j)$ give a spectral triple for the harmonic Sierpinski gasket (by applying Proposition 1 in \cite{lapsar})
$$S(K_H) = \left(C(K_H), \ \ \bigoplus_{j\geq1} \Hil_{\ell_j}, \ \ D_{K_H} := \bigoplus_{j\geq1}D_{\ell_j}\right)$$
where $\ell_j = L(\Phi(R_j))$ and the representation is given by $\pi_H = \bigoplus_j\pi_{\ell_j}$. 

We begin by showing that the spectral dimension $\ds_H = \ds(K_H)$ of $K_H$ is finite. A direct computation or an application of Proposition 1 in \cite{lapsar} gives that 
$$\ds_H = \inf\left\{ s>1 : \sum_{j\geq 1} \ell_j^s <\infty\right\},$$
from which it follows that $\ds_H \geq 1$; however, one must show that $\ds_H <\infty$. In \cite{kig3} Kigami obtains bounds for the lengths $\ell_j=  L(\Phi(R_j))$. We will use these results in the lemma that follows.  


\begin{lemma}
The sum 
$$\sum_{j\geq 1} \ell_j^s$$
where $\ell_j = L(\Phi(R_j))$ converges for $s> \frac{\log{3}}{\log{5} - \log{3}} \approx 2.151$. In particular, 
$$1\leq \ds_H \leq \frac{\log{3}}{\log{5} - \log{3}}.$$
\end{lemma}
\begin{proof}
Let $p, q \in \Phi(f_w(V_0))$ for some word $w$ of length $|w| = m$ and let $\Phi(R_j)$ be the curve in $K_H$ which connects $p$ and $q$. 
By Lemma 5.6 in \cite{kig3},
$$\frac{2}{5}\  \diam(\Phi(f_w(T))) \leq L(\Phi(R_j)) \leq 2 \ \diam(\Phi(f_w(T))).$$
Note that
\begin{align*}
\diam(\Phi(f_w(T))) &= \sup\{ |\Phi(f_w(x)) -\Phi(f_w(y))| : x, y\in T\}\\
& = \sup\{|H_w(\Phi(x)) - H_w(\Phi(y))|: x, y \in T\}\\
& \leq \left(\frac{3}{5}\right)^m \sup\{|\Phi(x)- \Phi(y)|: x, y\in T\},
\end{align*}
so $\diam(\Phi(f_w(T))) \leq c\left(\frac{3}{5}\right)^m$, where $c$ is some constant not depending on $w$. 
Then
\begin{align}
\sum_{j=1}^\infty \ell_j^s &\leq \sum_{m=0}^\infty 3^{m+1} 2c \left(\frac{3}{5}\right)^{ms}\\
& = 6c\sum_{m=0}^\infty \left( \frac{3^{s+1}}{5^s}\right)^m\\
& = 6c\ \frac{5^s}{5^s - 3^{s+1}},\label{sumconv}
\end{align}
where we have assumed that $s> \frac{\log{3}}{\log{5} - \log{3}}$ in equality (\ref{sumconv}). From this and Proposition 1 in \cite{lapsar} we have that 
$$1\leq \ds_h \leq \frac{\log{3}}{\log{5} - \log{3}}.$$
\end{proof}

According to the previously mentioned results of Alain Connes in \cite{co}, the map $\trw( \pi_H(\cdot) |D_{K_H}|^{-\ds_H})$ is then a positive linear functional on $C(K_H)$.

For $n\geq 0$ and $k\in \{1, 2, \dots, 3^n\}$, write $\Delta_{n, k}$ for the $3^n$ triangles in $SG_n$ and for $j\in \{1, 2, 3\}$ write $x_{n, k, j}$ for the midpoints of the edges of these triangles. For $n\geq0$, define a positive linear functional $\psi_n: C(SG) \to \C$ of norm 1, by
$$\psi_n(f) = \frac{1}{3^{n+1}}\sum_{k=1}^{3^n} \sum_{j=1}^3 f(x_{n, k, j}).$$
In Proposition 8.6 of \cite{cil} it was shown that the sequence $\{\psi_n\}$ converges in the weak-$^\ast$ topology on the dual of $C(SG)$ to the positive linear functional $\psi$ given by
$$\psi(f) := \int_{SG} f \ d\calH\  ,$$
where $\calH$ is the $\frac{\log3}{\log 2}$-dimensional Hausdorff probability measure on $SG$. 

Recall that the map $\Phi:SG \to K_H$ is a homeomorphism when we give $SG$ and $K_H$ the topology induced by the Euclidean metric in $\R^2$ and $\R^3$, respectively. In $SG$ the Euclidean metric and the geodesic metric are equivalent, but in $K_H$ this is not the case \cite{kig3}. However, one can say that the geodesic metric on $K_H$, denoted $d_{geo}(\cdot, \cdot)$, satisfies $ |\cdot| \leq d_{geo}(\cdot, \cdot)$ where $|\cdot|$ is the Euclidean metric. Then $\Phi: (SG, |\cdot|) \to (K_H, d_{geo})$ is still a bijection and $\Phi^{-1}$ is a continuous map. From here forward, we will endow the harmonic Sierpinski gasket with the geodesic metric. 

\begin{lemma}\label{contfunc}
If $h\in C(K_H)$, then $h\circ \Phi\in C(SG)$. 
\end{lemma}
\begin{proof}
Let $\ep>0$ and $h \in C(K_H)$. Then there is a $\delta>0$ such that $ d_{geo}(\phi(x), \phi(y)) <\delta$ implies $|h\circ \Phi(x) -h\circ \Phi(y)| <\ep.$ Since the perimeter of the ``triangles", $\Phi(\Delta_{n, j})$, goes to zero as $n$ grows, we can choose an $n_0$ large enough so that the perimeter of $\Phi(\Delta_{n_0, j})$ is small enough and $d_{geo}(\phi(x), \phi(y)) <\delta$ for $x,y$ in the portion of $SG$ contained within $\Delta_{n_0, j}$. It follows that $h\circ \Phi$ is continuous from $(SG, |\cdot|)$ to $\R$. 
\end{proof}

Let $\tilde \psi_n$ be the positive linear functional on $C(K_H)$ given by 
$$ \tilde\psi_n(h) = \frac{1}{3^{n+1}}\sum_{k=1}^{3^n} \sum_{j=1}^3 h(\Phi(x_{n, k, j})) = \psi_n(h\circ \Phi),$$
where $h\in C(K_H)$.

\begin{proposition}
The sequence $\{\tilde\psi_n\}$ converges in the weak-$^\ast$topology on the dual of $C(K_H)$ to the positive linear functional given by
$$\tilde\psi(h) := \int_{SG} h\circ\Phi(x) \ d\h(x) = \int_{K_H} h(y) \ d (\h\circ \Phi^{-1})(y),$$
where $\calH$ is the $\log_23$-Hausdorff probability measure on $SG$ and $h\in C(K_H)$. 
Also, $\tilde \psi$ has the property
$$\tilde\psi (h) = \frac{1}{3}\sum_{j=1}^3 \tilde\psi (h\circ H_j),$$
where $h\in C(K_H)$ and $H_j$ for $j = 1, 2, 3$ are the affine maps which determine $K_H$. 
\end{proposition}

\begin{proof}
That $\tilde\psi_n \to \tilde\psi$ follows from the fact that $\psi_n \to\psi$ and that, according to Lemma \ref{contfunc}, $h\circ \Phi \in C(SG)$ whenever $h\in C(K_H)$.

To see that $\tilde\psi$ satisfies the stated property, note that the condition $\tilde\psi (h) = \frac{1}{3}\sum_{j=1}^3 \tilde\psi (h\circ H_j)$ is the same as 
$$\int_{SG} h\circ \Phi \ d\h = \frac{1}{3}\sum_{j=1}^3 \int_{SG} h\circ H_j \circ\Phi \ d\h$$ 
and since $H_j \circ \Phi = \Phi \circ f_j$, where $f_j$ are the similarities defining $SG$, this condition is the same as
$$\int_{SG} h\circ \Phi \ d\h = \frac{1}{3}\sum_{j=1}^3 \int_{SG} h\circ \Phi\circ f_j \ d\h.$$ 
This condition holds since $h\circ \Phi \in C(SG)$ whenever $h\in C(K_H)$ and since $\h$ is the unique self-similar measure on $SG$ satisfying,
$$\int_{SG} g \ d\h = \frac{1}{3}\sum_{j=1}^3 \int_{SG} g \circ f_j \ d\h \ \ \ \ \ \text{ for all } g\in C(SG).$$
\end{proof}

We can now use this spectral triple to recover the standard self-affine measure on $K_H$. Self-affine measures such as this are described by Hutchinson in \cite{hut}. 

\begin{proposition}
Let $\tau: C(K_H) \to \C$ be given by $\tau(h):= \trw (\pi_H(h)|D_{K_H}|^{-\ds_H})$. Then 
$$\tau(h) =  \trw (\pi_H(h)|D_{K_H}|^{-\ds_H})= c\int_{K_H} h(x)\ d\mu,$$
where 
$\mu$ is the unique self-affine measure on $K_H$ satisfying,
$$\int h \ d\mu = \frac{1}{3}\sum_{j=1}^3\int (h\circ H_j) d\mu \ \ \ \ \text{ for each } f \in C(K_H).$$
\end{proposition}

\begin{proof}
Let $h \in C(K_H)$ and $\ep>0$. Choose $n_0 \in\N$ such that for any $k \in\{1, 2, \dots, 3^{n_0}\}$ and $x, y$ inside or on the ``triangle", $\Phi(\Delta_{n_0, k}),$ we have $|h(x)-h(y)| <\ep$. Let $n>n_0$ and define
$$v_{n_0, k}^n (h) = \frac{1}{3^{(n-n_0)+1}} \sum_{i_k}\sum_{j=1}^3 h(\Phi(x_{n, i_k, j})),$$
where the $x_{n, i_k, j}$ are the midpoints of the edges in the triangles in the $n$-th step construction of the gasket, $SG_n$, which are contained in or on the border of $\Delta_{n_0, k}$. Note the dependence of $i_k$ on $k$ and that the number of terms in the sum $\sum_{i_k}$ is precisely $3^{n-n_0}$.
Denote by $\Phi(SG_{n_0, k})$ the image under $\Phi$ of the portion of $SG$ in $\Delta_{n_0, k}$, and $I_{n_0, k}$ and $h_{n_0, k}$ for the restrictions of the functions $I= 1$ and $h$ on $K_H$ to $\Phi(SG_{n_0,k})$.
 
Notice that
\begin{align*}
 \left |v_{n_0, k}^n (h)I_{n_0,k} - h_{n_0, k}\right| 
&= \left|\frac{1}{3^{(n-n_0)+1}} \sum_{i_k}\sum_{j=1}^3 h(\Phi(x_{n, i_k, j}))I_{n_0,k}-\frac{1}{3^{(n-n_0)+1}} \sum_{i_k}\sum_{j=1}^3 h_{n_0,k}\right|\\
&\leq \frac{1}{3^{(n-n_0)+1}} \sum_{i_k}\sum_{j=1}^3 \left |h(\Phi(x_{n, i_k, j}))I_{n_0,k}- h_{n_0, k}\right|\\
&< \frac{1}{3^{(n-n_0)+1}} \sum_{i_k}\sum_{j=1}^3 \ep\\
&=\ep,
\end{align*}
so we have the inequalities,
$$- \ep I_{n_0, k}\leq v_{n_0, k}^n(h)I_{n_0, k} -  h_{n_0,k}\leq \ep I_{n_0, k}.$$
Now, for each space $\Phi(SG_{n_0, k})$, we can define a spectral triple by deleting all summands in $S(K_H)$ which correspond to an edge not in $\Phi(SG_{n_0, k})$. For such a triple we get the corresponding functional $\tau_{n_0, k}$. By the linearity of the Dixmier trace and the fact that as operators $\pi_H(h) = \sum_{k=1}^{3^{n_0}} \pi_{n_0, k}(h)$ (where $\pi_{n_0, k}$ are the representations corresponding to the triples for $\Phi(SG_{n_0, k})$), we have
$$\tau(h) = \sum_{k=1}^{3^{n_0}} \tau_{n_0, k}(h_{n_0,k}) \text{ and } \tau_{n_0, k}(I_{n_0, k})= 3^{-n_0}\tau(I).$$
As $\tau$ is a positive linear functional and hence preserves order, we have
$$- \ep \ \tau_{n_0, k}(I_{n_0, k})\leq v_{n_0, k}^n(h)\tau_{n_0, k}(I_{n_0, k}) -  \tau_{n_0, k}(h_{n_0,k})\leq \ep\ \tau_{n_0, k}(I_{n_0, k})$$ 
and hence
$$- \ep \ 3^{-n_0}\tau(I) \leq v_{n_0, k}^n(h)3^{-n_0}\tau(I)-  \tau_{n_0, k}(h_{n_0,k})\leq \ep \ 3^{-n_0}\tau(I).$$
Summing over $k\in\{1, 2, \dots, 3^{n_0}\},$ we get 
$$- \ep\ \tau(I)\leq 3^{-n_0}\sum_{k=1}^{3^{n_0}}v_{n_0, k}^n(h)\tau(I) - \tau(h)\leq \ep\ \tau(I)$$ 
and using the fact that $\tilde \psi_n = 3^{-n_0}\sum_{k=1}^{3^{n_0}} v_{n_0, k}^n,$ we find
$$- \ep\ \tau(I)\leq \tilde\psi_n(h)\tau(I) - \tau(h)\leq \ep\ \tau(I).$$ 
Letting $n\to \infty$ we have 
$$- \ep\ \tau(I)\leq \tilde\psi(h)\tau(I) - \tau(h)\leq \ep\ \tau(I)$$ 
and hence $\left|\tau(I)\tilde\psi(h)-\tau(h)\right|< \tau(I)\ep$. This gives
$$\trw (\pi_H(h)|D_{K_H}|^{-\ds_H})= c\int_{K_H} h(x)\ d\mu\ ,$$ 
where $c = \tau(I)$.
\end{proof}

It was previously conjectured that the Dixmier trace on $K_H$ would recover the Hausdorff measure on $(K_H, d_{geo})$; however, the Dixmier trace recovers the self-affine measure of weights $1/3$ and it can be shown that this self-affine measure is not the same as the Hausdorff measure on $K_H$, \cite{kaj1}. Briefly, the value of $\mu$, the self-affine measure of weights $1/3$, on sets of the form $H_w(K_H)$ where $H_w = H_{w_1w_2\dots w_k}$, is given by $\mu(H_w(K_H)) = \left(\frac{1}{3}\right)^{|w|} \mu(K_H) = \left(\frac{1}{3}\right)^{|w|}$. This means the value of $\mu$ on a set like $H_w(K_H)$ is completely determined by the \emph{length} of the word $w$. It was shown by Kajino in Proposition 6.4 of \cite{kaj} that there exist positive constants $c_1, c_2$ such that
$$c_1\|M_w\|^d \leq \calH^d(H_w(K_H)) \leq c_2 \|M_w\|^d,$$ 
where $d$ is the Hausdorff dimension of $(K_H, d_{geo})$, $M_w = M_{w_1}\cdots M_{w_k}$, and the $M_{w_i}$ are the matrices in the definition of the maps $H_j: K_H\to K_H$ which determine $K_H$. Changing the word $w$ can drastically change the norm of the matrices $M_w$ and hence the value of 
$\calH^d(H_w(K_H))$. With a bit more work, these facts show that the self-affine measure is not the same as the Hausdorff measure and hence this construction of a spectral triple for $K_H$ cannot recover the Hausdorff measure. 
It would be interesting to see what kind of spectral triple on $K_H$ would recover the Hausdorff measure.


\section{Spectral triple for the stretched Sierpinski gasket, $K_\al$.}

In this section we will consider the direct sum curve triple for the stretched Sierpinski gasket and show that it recovers the Hausdorff dimension, the geodesic metric, and the Hausdorff measure on $K_\al$. These results are of interest since the space $K_\al$ is a self-affine space as opposed to a self-similar space. In general, self-affine spaces are more difficult to study than their structure rich self-similar sisters. 

Let us introduce some notation. The notation will be similar to that used for the Sierpinski gasket, but will include a superscript $s$ to indicate that we are working with the \emph{stretched} Sierpinski gasket. Let $p_1 = (0, 0), \ p_2 = \left(\frac{1}{2}, \frac{\sqrt{3}}{2}\right)$, and $p_3 = (1,0)$ as before. Define,
$$V^s_0 := \{p_1, p_2, p_3\} \ \ \ \text{ and for } n\geq1, \text{ let } \ \ \ V^s_n := \bigcup_{w\in\{1, 2, 3\}^n} F_w(\{p_1, p_2, p_3\}),$$
where $w = w_1w_2\cdots w_n\in \{1, 2, 3\}^n$, $F_w = F_{w_n}\circ \cdots \circ F_{w_2} \circ F_{w_1}$, and the $F_j$'s are as in (\ref{defssg}).
These are the vertices of the triangles in the approximations of the stretched Sierpinski gasket. Let $V^{s\ast} := \bigcup_{n\geq0}V^s_n$, the set of all vertices of the triangles in $K_\al$. 

It will be important to distinguish between the two different types of edges in the graph approximations of $K_\al$, namely the triangle edges and the edges joining the triangles. For $x, y\in \R^2$, the symbols $[x\to y]$ or $(x \to y)$ refer to the line segment in $\R^2$ connecting $x$ and $y$ which include or exclude the points $x$ and $y$, respectively. Define
$$T_0 := \{ [p_j \to p_i]: i, j= 1, 2, 3 \text{ and }  j\neq i\}$$ 
(the edges in the outer triangle) and for $n\geq1$,
$$T_n := \{ [x \to y]:  \exists \ w\in\{1, 2, 3\}^n \text{ such that } x, y \in F_w(V^s_0)\}$$
(edges in the triangles at level $n$).
The set $T_n$ is the collection of triangle edges in the $n$-th level approximation of the stretched Sierpinski gasket. Recall the notation for the joining edges in $K_\al$: $J_0 = \emptyset$ and for $n\geq1$
$$ J_n = \bigcup_{m=0}^{n-1} \bigcup_{w\in \{1, 2, 3\}^m} F_w\left( \bigcup_{i=1}^3 e_i\right),$$
where $e_1, e_2, e_3$ are the three initial joining edges. Also, $J^* = \cup_{n\geq1} J_n$. 

We would like to distinguish between the collection of \emph{points} in $K_\al$ which lie in the sets $J_n$ and the collection of \emph{edges} that make up the set $J_n$. 
Write $\mathcal{J}_n$ for the collection of joining edges at stage $n$, which include the endpoints:
$$ \mathcal{J}_n= \bigcup_{m=0}^{n-1} \bigcup_{w\in \{1, 2, 3\}^m} \{F_w\left(\overline{e_i}\right): i=1, 2, 3\}\ \ \ \ \ \ \ \ \ \ \ \ \  \text{ for } n\geq1$$
and $\calJ^\ast = \cup_{n\geq1} \calJ_n$.
Finally, define $\mathscr{E}_n := T_n \cup \mathcal{J}_n$.

For each $\ep = [\ep^-\to \ep^+] \in \mathscr{E}_n$, where $\ep^+, \ep^- \in\R^2$ denote the endpoints of the edge $\ep$, define $R^s_{\ep}:[0, L(\ep)] \to\R^2$ by 
$$R^s_\ep(t) = \frac{1}{L(\ep)}\left(\ep^+t+(L(\ep) -t)\ep^-\right),$$ 
where $L(\ep)$ denotes the length of the edge $\ep$. 

It was shown in \cite{ar} that
$$K_\al = \overline{\bigcup_{n\geq 0}\bigcup_{\ep\in \mathscr{E}_n} R^s_\ep([0, L(\ep)])},$$ 
where the closure is taken with respect to the Euclidean metric. It was also shown in \cite{ar} that the Euclidean metric, the effective resistance metric, and the geodesic metric on $K_\al$ are all equivalent. This means $K_\al$ satisfies Axiom 1, where the curves are the $R^s_\ep$ corresponding to the edges in the sets $\scrE_n$. It follows from the results in \cite{lapsar} mentioned previously that the direct sum of the $R^s_\ep$ curve triples is a spectral triple for $K_\al$. Denote this spectral triple by 
$$S(K_\al) =(C(K_\al), \ \Hil_\al, \ D_\al),$$ 
with representation $\pi_\al: C(K_\al) \to \mathcal{B}(\Hil_\al)$.

\subsection{Recovery of the Hausdorff dimension and geodesic metric on $K_\al$}
It was shown in \cite{ar} that the Hausdorff dimension of the stretched Sierpinski gasket of parameter $\al$ is
$$d_\al := \dfrac{\log(3)}{\log(2) - \log(1-\al)}.$$
We begin the section by showing that the spectral triple $S(K_\al)$ recovers the Hausdorff dimension of $K_\al$. First let us enumerate the edges in the set 
$\scrE := \bigcup_{n\geq0} \scrE_n$ and write $\scrE = \{\ep_1, \ep_2, \dots\}$. To simplify notation we write $R^s_j=R^s_{\ep_j}$. 

\begin{proposition}\label{newtr}
For $p> 1$ and each fixed $j\geq 1$, 
$$\text{tr}(|D_j|^{-p}) = \beta_p l_j^{p},$$
where $D_j$ is the operator in the spectral triple for the edge $R^s_j$, $L(R^s_j) = l_j$, and $\beta_p =  \dfrac{2^{p+1}(1-2^{-p})\zeta(p)}{\pi^p}$. 
Furthermore, for $p>1$, 
$$\text{tr}(|D_\al|^{-p}) =\beta_p \sum_{j=1}^\infty l_j^{p},$$ 
where $D_\al$ is the operator in the spectral triple for $K_\al$.
If $p >d_\al$, we have 
$$\text{tr}(|D_\al|^{-p}) = \dfrac{\beta_p2^p(3+3\al^p)}{2^p - 3(1-\al)^p}.$$ 
\end{proposition}
\begin{proof}
Recall that the eigenvalues of the operator $D_\al$ are given by
$$\bigcup_{j\geq1} \left\{\frac{(2k+1)\pi}{2l_j} : k\in \Z\right\}.$$
The values $L(R^s_j) = l_j$ are in the set 
$$\bigcup_{n\geq0}\left \{ \left(\frac{1-\al}{2}\right)^n,  \ \ \ \al\left(\frac{1-\al}{2}\right)^n\right \},$$ 
with multiplicity $3^{n+1}$ for each length like $\left(\frac{1-\al}{2}\right)^n$ or like $\al\left(\frac{1-\al}{2}\right)^n$. 
Assuming $p>1$, we have
$$\text{tr} (|D_j|^{-p}) = \sum_{k\in\Z} \left| \dfrac{(2k+1)\pi}{2l_j}\right|^{-p} 
= \dfrac{2^{p+1}l_j^p}{\pi^p}\sum_{k=0}^\infty \dfrac{1}{|2k+1|^p}
 = \dfrac{2^{p+1}l_j^p}{\pi^p}(1-2^{-p})\zeta(p)
= \beta_p l_j^p$$
and
\begin{align}
\text{tr} (|D_\al|^{-p})  & = \beta_p\sum_{j=1}^\infty l_j^p\\
&= \beta_p\left(\sum_{n=0}^\infty 3^{n+1}\left(\frac{1-\al}{2}\right)^{np} +\sum_{m=0}^\infty 3^{m+1} \al^p\left(\frac{1-\al}{2}\right)^{mp}\right)\\
&= \beta_p\left(3\sum_{n=0}^\infty \left(3\left(\frac{1-\al}{2}\right)^p\right)^{n} +3\al^p\sum_{m=0}^\infty \left(3\left(\frac{1-\al}{2}\right)^{p}\right)^m \right)\\
&=  \beta_p(3+3\al^p) \dfrac{2^p}{2^p - 3(1-\al)^p}\ ,\label{eq1}
\end{align}
where (\ref{eq1}) requires the further assumption that $p>d_\al$.
\end{proof} 

\begin{corollary}\label{newdim}
The spectral dimension, $\ds(K_\al)$, induced by the spectral triple $S(K_\al)$ for $K_\al$ is equal to $d_\al=\frac{\log(3)}{\log(2) - \log(1-\al)}$, the Hausdorff dimension of the stretched Sierpinski gasket of parameter $\al$. 
\end{corollary}
\begin{proof}
An application of the limit comparison test will show that computing the abscissa of convergence of the series $\text{tr}((1+D_\al^2)^{-p/2})$ is the same as computing the abscissa of convergence of the series $\text{tr}(|D_\al|^{-p})$. It was shown in \cite{ar} that the Hausdorff dimension of $K_H$ is given by $d_\al=\frac{\log(3)}{\log(2) - \log(1-\al)}$. In Proposition \ref{newtr} we found that the abscissa of convergence of the series $\text{tr}(|D_\al|^{-p})$ is $d_\al$. It follows that 
$$\ds(K_\al)= \frac{\log(3)}{\log(2)-\log(1-\al)}.$$
\end{proof} 


Thus the spectral triple $S(K_\al)$ recovers the Hausdorff dimension of $K_\al$. Next we recover the geodesic metric on $K_\al$ by using the spectral metric induced by $S(K_\al)$.

Recall that $K_\al = W_\al \cup J^*$  and since the set $V^{s*}$ is dense in $W_\al$, the set $V^{s*} \cup J^*$ is dense in $K_\al$. 


\begin{proposition}\label{newax1} 
For any $p\in V^{s*}\cup J^*$ and any $q\in K_\al$, there is a path of minimal length from $p$ to $q$ which is a concatenation of (finite or countably many) triangle edges, joining edges, or segments of joining edges at the start or end of the path (possibly both). 
\end{proposition}
\begin{proof}
(Case $p\in V^{s*}$ and $q\in W_\al$) \\
Let $p\in V^{s*}$ and $q\in W_\al$. Let $m$ be the smallest integer such that $p$ and $q$ are in different $m$ cells, $F_w(W_\al)$ and $F_{v}(W_\al)$, where $|w| = |v| = m$. Suppose further that $p$ is an $m$-vertex for an edge in $F_w(W_\al)$ (and not just an endpoint for some further approximation). We are considering the space $K_\al$ with the metric $d_{geo}$ which is equivalent to the Euclidean metric. By the Hopf--Rinow theorem we know that $K_\al$ has minimizing geodesics. Let $\gamma$ be a minimal path between $p$ and $q$. Then there is a vertex $v_1$ of an edge in $F_{v}(W_\al)$ such that $\gamma$ passes through $v_1$. The portion of $\gamma$ which connects $p$ and $v_1$ must look like one of
$$\ep_j, \ \ \ \  \ \ep_j \ast \ep_t, \ \ \ \ \ \ep_j \ast \ep_t\ast \ep_{j'}, \ \ \ \ \  \ep_t, \ \ \ \ \ \ep_{t} \ast \ep_j, \ \ \ \ \ \ep_t\ast \ep_j \ast \ep_{t'},$$
where $\ep_j, \ep_{j'}\in \mathcal{J}_m$, $\ep_t, \ep_{t'} \in T_m$, and $\ep\ast \beta$ denotes the concatenation of the edges $\ep$ and $\beta$. Note that here, $\ep \ast \beta$ means that we first travel along the edge $\ep$ and then along the edge $\beta$. 
Note that if $p$ and $v_1$ can be joined by a path with one or two edges, then that path is unique of minimal length. There may be more than one path with three edges connecting $p$ and $v_1$ and these will look like $ \ep_j \ast \ep_t \ast \ep_{j'}$ and $\ep_t \ast \ep_j \ast \ep_{t'}$. In this case we take the shorter of the two paths, namely $\ep_{j} \ast \ep_t \ast \ep_{j'}$, which will be the unique path of minimal length. No path between $p$ and $v_1$ with four or more edges will be minimal. 

Write $\gamma_1$ for the concatenation of the edges connecting $p$ and $v_1$. Repeating this argument with $v_1$ and $q$, we get a unique path $\gamma_2$ of minimal length from $v_1$ to some vertex $v_2$ in an $m'$ cell $F_{w'}(W_\al)$, where $|w'| = m'>m$ and $q\in F_{w'}(W_\al)$. In this way, we get a countable concatenation of paths $\gamma_i$ whose lengths go to zero since the lengths of the edges making them up go to zero. Then we must have that $\gamma =  \overline{\cup_{i\geq1} \gamma_i}$ is the path of minimal length from $p$ to $q$. 

If $p$ is not an $m$-vertex in $F_w(W_\al)$, then choose an $m$-vertex, $u$, which lies on a shortest path between $p$ and $q$ and apply the previous argument on $u, p$ and $u, q$. We then reverse the path (which will consist of finitely many edges since $p$ is a vertex) between $u$ and $p$ and concatenate with the path between $u$ and $q$ to get a path from $p$ to $q$. 
\\

\noindent(Case $p\in V^{s*}$ and $q\in J^*$)\\
 If $q \in \ep_j$ for some $\ep_j \in \calJ^*$, then apply the argument above to get a path, $\gamma$, from $p$ to one of the endpoints $\ep_j^-$ or $\ep_j^+$ (whichever is closest to $p$ and hence yields the shortest path $\gamma$). Concatenate $\gamma$ with the line segment between $\ep_j^-$ or $\ep_j^+$ and $q$. This again yields a unique shortest path between $p$ and $q$. 
\\

\noindent(Case $p\in J^*$ and $q\in K_\al$)\\
 Suppose now that $p \in \ep_j$ for some $\ep_j \in \calJ^*$ and $q\in K_\al$. Without loss of generality, assume $\ep_j^-$ is closer to $q$ than  $\ep_j^+$. Apply the above argument to get a path between $\ep_j^-$ and $q$ and concatenate with the line segment connecting $p$ to $\ep_j^-$. This concludes the proof. 
\end{proof} 

In the following lemma we use the notation
$$\lip(f) = \sup\left\{\frac{|f(x)-f(y)|}{|x-y|}: x\neq y \in \R \right\}$$
for the Lipschitz seminorm of a function $f: \R \to \R^n$.


\begin{lemma}\label{lembddifflip}
Let $r: [0, \ell] \to X \subseteq \R^n$ be a curve and consider its spectral triple $(C(X), \ \Hil_\ell, \ D_\ell)$. For $f\in C(X)$, we have that $[D_\ell, \pi_\ell(f)]$ is bounded if and only if $f\circ r(|x|)$ is Lipschitz if and only if $f\circ r(|x|)$ is differentiable almost everywhere.
\end{lemma}
\begin{proof}
It is well known that a function $h: [0, \ell] \to \R^n$ is Lipschitz if and only if it is differentiable almost everywhere. Hence, it only remains to prove the first equivalence.

Suppose $f\in C(X)$ and $[D_\ell, \pi_\ell(f)]$ is bounded. For $g = 1\in L^2[-\ell, \ell]$, 
$$\|[D_\ell, \pi_\ell(f)]g\|_2 \leq \|[D_\ell, \pi_\ell(f)]\| \|g\|_2 = \|[D_\ell, \pi_\ell(f)]\|< \infty$$ and since
$$[D_\ell, \pi_\ell(f)]g = D_\ell\pi_\ell(f)g - \pi_\ell(f)D_\ell(g) = D_\ell( f\circ r(|x|)),$$
 it follows that $f\circ r (|x|)\in \text{Dom}(D_\ell)$. By Lemma \ref{lemfromcil}, there exists a bounded measurable function $g$ such that
 $$|f\circ r(|x|) - f\circ r(|y|)| = \left | \int_{|y|}^{|x|} g(t) dt\right|\leq  \int_{|y|}^{|x|} |g(t)| dt\leq \|g\|_\infty ||x|-|y||.$$
 This shows that $f\circ r(|x|)$ is Lipschitz and $\lip(f) \leq \|g\|_\infty = \|D f\|_\infty$.
 
Suppose now that $f\circ r(|x|)$ is Lipschitz and hence is differentiable almost everywhere. Then
$$[D_\ell, \pi_\ell(f)] g = \pi_\ell(D f) g$$
for $g\in C^1[-\ell, \ell]$. Thus $[D_\ell, \pi_\ell(f)]$ is densely defined and can be extended to the bounded operator $\pi_\ell(D f)$ on $L^2[-\ell, \ell]$.
\end{proof}

\begin{definition}
Let $\lip_g(\cdot)$ be the Lipschitz seminorm for the compact metric space $(K_\al, d_{geo})$ given by
$$\lip_g(f) = \sup\left\{ \frac{|f(x)-f(y)|}{d_{geo}(x, y)}: x, y\in K_\al, x\neq y \right\}.$$
\end{definition}

Note that since the geodesic metric and the Euclidean metric on $K_\al$ are equivalent, if $f$ is Lipschitz with respect to the Euclidean metric (i.e. $\lip(f) <\infty$), then $f$ is Lipschitz with respect to the geodesic metric (i.e. $\lip_g(f) < \infty$) and conversely. 

\begin{proposition}\label{newlip}
For any function $f\in C(K_\al)$ such that $\|[D_\al, \pi_\al(f)]\|< \infty$, 
$$ \|D_\al f\|_{\infty, K_\al} = \lip_g(f).$$
\end{proposition}
\begin{proof}
By Lemma \ref{lembddifflip}, if $\|[D_\al, \pi_\al(f)]\|< \infty$ then $f$ is Lipschitz and differentiable almost everywhere. 
Since $K = \overline{\cup_{j\geq1} R^s_j}$,
\begin{align*}
\|D_\al f\|_{\infty, K_\al} &= \sup_{j} \{ \|D_jf\|_{\infty, R^s_j}\}\\
& = \sup_{j} \left\{ \left\|-i \frac{\partial f}{\partial x}\right\|_{\infty, R^s_j}\right\}\\
& = \sup_{j} \left\{ \sup_{p, q\in R^s_j} \left \{ \frac{|f(p) - f(q)|}{d_{geo}(p,q)}\right\}\right\}\\
&\leq \lip_g(f).
\end{align*}

For the reverse inequality first suppose $p\in V^{s*}\cup J^*$ and $q\in K_\al$. Then by Proposition \ref{newax1} there is a minimizing geodesic between $p$ and $q$ made of $R^s_j$ curves (i.e. triangle or joining edges) and segments of joining edges. Suppose first that the geodesic consists only of complete $R^s_j$ curves. Let $(p_k, p_{k+1})$ track the endpoints of these $R^s_j$, so $p=p_1$ and $\lim_{k\to \infty} p_k = q$. Then
\begin{align*}
|f(p) - f(p_k)| &\leq \sum_{j=1}^{k-1}|f(p_j) - f(p_{j+1})| \\
&\leq \sum_{j=1}^{k-1} d_{geo}(p_j, p_{j+1}) \|D_jf\|_{\infty, R_j}\\
&\leq \|D_\al f\|_{\infty, K_\al} \sum_{j=1}^{k-1} d_{geo}(p_j, p_{j+1}) \\
&= \|D_\al f\|_{\infty,K_\al}  d_{geo}(p, p_{k}) \\
\end{align*}
and by continuity of $f$ and $d_{geo}(p, x)$, 
$$\frac{|f(p) - f(q)|}{d_{geo}(p, q)} \leq \|D_\al f\|_{\infty, K_\al}.$$
Now suppose the minimizing geodesic between $p$ and $q$ looks like $\gamma_1 \ast \{R^s_{j_k}\} \ast \gamma_2$, where the $\gamma_1, \gamma_2$ are segments of joining edges (which can be empty or and entire edge) and are being concatenated with either finitely many edges, $\{R^s_{j_k}\}_{k \in \{1, 2, \dots, N\}}$, or infinitely many edges, $\{R^s_{j_k}\}_{k\in \{1, 2, \dots\}}$. We will allow for $\gamma_1$ to be an entire edge but will assume that it is nonempty. On the other hand, we may assume that $\gamma_2$ is not an entire edge, since otherwise it would be included as one of the $R^s_{j_k}$ curves, but $\gamma_2$ may be empty. In fact, if the number of curves $R^s_{j_k}$ is countably infinite then $\gamma_2 = \emptyset$. 

Set $p_0 = p \in \gamma_1$ and for $k\geq 1$, let $(p_k, p_{k+1})$ track the endpoints of the $R^s_{j_k}$ curves. Let $R^s_{j_0}$ be the edge on which $p$ lies (if $\gamma_1$ is an entire edge, then $R^s_{j_0} = \gamma_1$) and note that $p_1$ is an endpoint of both $R^s_{j_0}$ and $R^s_{j_1}$. If $\gamma_2 \neq\emptyset$, then the number of $R^s_{j_k}$ curves is finite; let $p_{N+1}$ be the endpoint which connects the last edge, $R^s_{j_N}$, to $\gamma_2$. Let $R^s_{j_{N+1}}$ be the edge with endpoint $p_{N+1}$ and containing $q$, and set $p_k = q$ for $k\geq N+1$. As before,
$$|f(p) - f(p_k+1)| \leq \|D_\al f\|_{\infty,K_\al}  d_{geo}(p, p_{k}),$$
where we have used the fact that for $k\geq 0$, the points $p_k, p_{k+1}$ are on the same edge $R^s_{j_k}$. Again by continuity, 
$$\frac{|f(p) - f(q)|}{d_{geo}(p, q)} \leq \|D_\al f\|_{\infty, K_\al}.$$ 
If $\gamma_2=\emptyset$, then repeat the above arguments with the sequence 
$$\{p_0=p\} \cup \{p_k: p_k \text{ endpoints of the edges }R^s_{j_k}, k\geq 1\},$$ 
to get the same estimate:
\begin{equation}
\frac{|f(p) - f(q)|}{d_{geo}(p, q)} \leq \|D_\al f\|_{\infty, K_\al},\label{est}
\end{equation}
which holds for any $p\in V^{s*}\cup J^*$ and $q\in K_\al$. 

Now let $p, q$ be arbitrary points in $K_\al$. Let $\gamma$ be a minimizing geodesic between $p$ and $q$. Since the set $V^{s*}\cup J^*$ is dense in $K_\al$, the path $\gamma$ intersects $V^{s*}\cup J^*$ at some point $r$. Let $\gamma_1$ be a minimizing geodesic between $p$ and $r$ and $ \gamma_2$ a minimizing geodesic between $r$ and $q$, where $\gamma_1, \gamma_2$ are made up of $R^s_j$ edges and portions of $R^s_j$ edges. The lengths of $\gamma_1$ and $\gamma_2$ must be the same as the lengths of the portions of $\gamma$ connecting $p$ and $r$, and $r$ and $q$. Let $\gamma_2$ be tracked by points $\{r_i\}$, where these points are endpoints of some $R^s_j$ edges or possibly points on a joining edge (as in the paths described above). Define $\gamma_{1i}$ to be the path obtained by concatenating the first $i$ parts of $\gamma_2$ with $\gamma_1$ at the point $r$. Using the estimate (\ref{est}) on the points $r_i\in V^{s*}\cup J^*$ and $p\in K_\al$ gives 
$$\frac{|f(p) - f(r_i)|}{d_{geo}(p, r_i)} \leq \|D_\al f\|_{\infty, K_\al}$$
and by continuity 
$$\frac{|f(p) - f(q)|}{d_{geo}(p, q)} \leq \|D_\al f\|_{\infty, K_\al}.$$ 
It now follows that $\lip_g(f) = \|D_\al f\|_{\infty,K_\al}$, as desired. 
\end{proof} 

\begin{theorem}
Let $d_{K_\al}(\cdot, \cdot)$ be the metric on $K_\al$ induced by the spectral triple $S(K_\al)$. Then for all $x, y\in K_\al$,
$$d_{K_\al}(x, y) = d_{geo}(x, y).$$
\end{theorem}	
\begin{proof}
The proof here relies on Proposition \ref{newlip} and is the same as the proof of Theorem 2 in \cite{lapsar}. We recreate it here, for the sake of completeness. 

Let $p, q\in K_\al$ and $f \in C(K_\al)$ such that $\|[D_\al, \pi_\al(f)]\|\leq 1$. By Lemma \ref{lemfromcil}, $f \in \text{Dom}(D_\al)$ and since $[D_\al, \pi_\al(f)]$ is bounded it must be the operator $\pi_\al(Df)$. Using that representations of $C^\ast$-algebras are isometries, we deduce that
$$\|D_\al f\|_\infty =  \|\pi_\al(D f)\| = \|[D_\al, \pi_\al(f)]\| \leq 1.$$ 
By Proposition \ref{newlip}, $\lip_g(f) = \|D_\al f\|_\infty\leq1$ and hence
$$\frac{|f(p)-f(q)|}{d_{geo}(p, q)} \leq 1;$$
so that $|f(p)-f(q)| \leq d_{geo}(p, q)$. This gives that $d_{K_\al}(p, q) \leq d_{geo}(p, q)$. 
For the reverse inequality consider the continuous function $h(x) = d_{geo}(x, q)$. Note that $\lip_g(h) = 1$ and hence, by Lemma \ref{lemfromcil} and Lemma  \ref{lembddifflip}, 
$\|[D_\al, \pi_\al (h)]\| \leq 1$.  Now since 
$$|h(p) - h(q)| = |0-d_{geo}(p, q)| = d_{geo}(p, q),$$
we have $d_{K_\al}(p, q) \leq d_{geo}(p, q)$ and hence $d_{K_\al}(p, q) = d_{geo}(p, q)$.
\end{proof}


\subsection{Recovery of the Hausdorff measure on $K_\al$}
In this section we show that the $d_\al$-dimensional Hausdorff measure, $\calH^{d_\al}$, is the unique self-affine measure satisfying 
$$\calH^{d_\al}(A) = \frac{1}{3} \sum_{i=1}^3 \calH^{d_\al} (F^{-1}_i(A))$$ 
for any Borel set $A \subseteq K_\al$. We then show that the measure defined by the Dixmier trace is the same as the $d_\al$-dimensional Hausdorff measure. Denote the Hausdorff dimension of a metric space $(X, d)$ by $\dim_H(X)$. 

Recall that $K_\al$ can be written in terms of its discrete part and its continuous part
$$K_\al = W_\al \cup J^*,$$ 
where this union is disjoint. Notice that for $A \subseteq K_\al$ it holds that $\dim_H(A \cap J^*) \leq \dim_{H} (J^*) = 1$ and $d_\al >1$ so $\calH^{d_\al}(A \cap J^*) =0$. This means
$$\calH^{d_\al}( A)=\calH^{d_\al}(A\cap W_\al)+\calH^{d_\al}(A \cap J^*) = \calH^{d_\al}(A\cap W_\al),$$
which shows that the $d_\al$-Hausdorff measure on $K_\al$ is the same as the $d_\al$-Hausdorff measure on $W_\al$.

The following is an easy consequence of the work in \cite{ar} and \cite{arf}. We give a proof, for the sake of completeness. 
\begin{proposition}\label{newselfsim}
The $d_\al$-dimensional Hausdorff measure on $K_\al$ satisfies the condition,
$$\calH^{d_\al}(A) = \frac{1}{3} \sum_{i=1}^3 \calH^{d_\al} (F^{-1}_i(A))$$ 
for any Borel set $A \subseteq K_\al$.
\end{proposition}
\begin{proof}
Let $A\subseteq K_\al$. Then 
\begin{equation}
A = F_1(A_1)\cup F_2(A_2)\cup F_3(A_3)\cup J \label{union}
\end{equation}
 where $J\subseteq J^*$ and the unions are disjoint. Then
$\calH^{d_\al}(A) = \sum_{j=1}^3\calH^{d_\al}(F_j(A_j))$. Note that since the maps $F_j$, for $j=1, 2, 3$, are similarities of parameter $\dfrac{1-\al}{2}$, for $U\subseteq K_\al$, it holds that 
$$\calH^{d_\al}(F_j(U)) = \left(\dfrac{1-\al}{2}\right)^{d_\al} \calH^{d_\al}(U), \ \ \ \ \ \ \ j=1, 2, 3$$ 
and since 
$$\left(\dfrac{1-\al}{2}\right)^{d_\al}  =  \left(\dfrac{1-\al}{2}\right)^{\frac{\log (3)}{\log\left(\frac{2}{1-\al}\right)}} 
= \left(\dfrac{1-\al}{2}\right)^{\log_{\frac{2}{1-\al}}(3)} = 3^{-1}$$
we have $\calH^{d_\al}(F_j(U)) = \dfrac{1}{3} \calH^{d_\al}(U)$ for $j=1, 2, 3$. Note that $F_j^{-1}(A) = A_j$ since the union in (\ref{union}) is disjoint and the $F_j$, for $j=1, 2, 3$, are injective. It then follows that
$$\frac{1}{3} \sum_{j=1}^3 \calH^{d_\al}(F^{-1}_j(A)) = \frac{1}{3} \sum_{j=1}^3 \calH^{d_\al}(A_j) = \sum_{j=1}^3 \calH^{d_\al} (F_j(A))\\
= \calH^{d_\al}(A),$$
as was to be shown.
\end{proof}

For $n\geq 1$ define the maps $\psi_{\al, n}: C(K_\al) \to \R$ by 
$$\psi_{\al, n}(f) = 2^{-1}3^{-n} \sum_{\ep\in \calJ_n\setminus \calJ_{n-1}} \sum_{s \in \{+, -\}} f(\ep^s),$$
where $\ep^-, \ep^+$ are the endpoints of the edge $\ep\in \calJ_n\setminus \calJ_{n-1}$. 

We will need the following notation. For $n>n_0 \geq0$, let 
$$\mathcal{S}^n_{n_0, h}  = \{ \ep \in \calJ_n\setminus \calJ_{n-1}: \ep\subset \Delta_{n_0, h}\},$$ 
where $\Delta_{n_0, h}$ is a triangle in the $n_0$-th step in the construction of the gasket and these triangles have been enumerated clockwise by $h\in\{1, 2, \dots, 3^{n_0}\}$. Let $e_{n_0, h}^\pm$ denote the two endpoints of edges in $\calJ_{n_0}$ which also lie in $\Delta_{n_0,h}$. Note that the points $e_{n_0, h}^\pm$ do not belong to the same edge in $\calJ_{n_0}$. We use the notation with superscript $\pm$ for convenience and not to indicate that these are the ``right" and ``left" endpoints of an edge, as is the case with the notation $\ep^\pm$. See Figure \ref{ssgedges}.

\begin{figure}
\begin{center}
\begin{tikzpicture}[scale = 1.5]

\node(a) at (.8, 1.9655) {\begingroup\makeatletter\def\f@size{10}\check@mathfonts
$e_{2, 4}^-$\endgroup};
\node(a) at (1.4, 1.4) {\begingroup\makeatletter\def\f@size{10}\check@mathfonts$e_{2, 4}^+$\endgroup};

\node(a) at (1, 2.33) {\begingroup\makeatletter\def\f@size{10}\check@mathfonts$e_{2, 5}^-$\endgroup};
\node(a) at (2, 2.3) {\begingroup\makeatletter\def\f@size{10}\check@mathfonts$e_{2, 5}^+$\endgroup};

\node(a) at (1.8, 1.4) {\begingroup\makeatletter\def\f@size{10}\check@mathfonts$e_{2, 6}^-$\endgroup};
\node(a) at (2.2, 1.9655) {\begingroup\makeatletter\def\f@size{10}\check@mathfonts$e_{2, 6}^+$\endgroup};

\node[circle,fill=black,inner sep=0pt,minimum size=3.5pt] (a) at (1.1485, 1.9895){};
\node[circle,fill=black,inner sep=0pt,minimum size=3.5pt] (a) at   (1.3595, 1.6238) {};

\node[circle,fill=black,inner sep=0pt,minimum size=3.5pt] (a) at  (1.289, 2.2325) {};
\node[circle,fill=black,inner sep=0pt,minimum size=3.5pt] (a) at   (1.7109, 2.2325) {};

\node[circle,fill=black,inner sep=0pt,minimum size=3.5pt] (a) at (1.641, 1.6238) {};
\node[circle,fill=black,inner sep=0pt,minimum size=3.5pt] (a) at (1.8515, 1.9895) {};

%

\node(a) at (-.1, .3655) {\begingroup\makeatletter\def\f@size{10}\check@mathfonts
$e_{2, 1}^-$\endgroup};
\node(a) at (.4, -.2) {\begingroup\makeatletter\def\f@size{10}\check@mathfonts$e_{2, 1}^+$\endgroup};

\node(a) at (.1, .73) {\begingroup\makeatletter\def\f@size{10}\check@mathfonts$e_{2, 2}^-$\endgroup};
\node(a) at (1, .73) {\begingroup\makeatletter\def\f@size{10}\check@mathfonts$e_{2, 2}^+$\endgroup};

\node(a) at (.8, -.2) {\begingroup\makeatletter\def\f@size{10}\check@mathfonts$e_{2, 3}^-$\endgroup};
\node(a) at (1.2, .3655) {\begingroup\makeatletter\def\f@size{10}\check@mathfonts$e_{2, 3}^+$\endgroup};

\node[circle,fill=black,inner sep=0pt,minimum size=3.5pt] (a) at (0.211, .3655) {};
\node[circle,fill=black,inner sep=0pt,minimum size=3.5pt] (a) at   (0.421875, 0) {};

\node[circle,fill=black,inner sep=0pt,minimum size=3.5pt] (a) at  (0.35155, .609) {};
\node[circle,fill=black,inner sep=0pt,minimum size=3.5pt] (a) at  (0.7735, .609) {};

\node[circle,fill=black,inner sep=0pt,minimum size=3.5pt] (a) at (0.703, 0) {};
\node[circle,fill=black,inner sep=0pt,minimum size=3.5pt] (a) at (0.914, .3655) {};

%

\node(a) at (1.8, .3655) {\begingroup\makeatletter\def\f@size{10}\check@mathfonts
$e_{2, 7}^-$\endgroup};
\node(a) at (2.3, -.2) {\begingroup\makeatletter\def\f@size{10}\check@mathfonts$e_{2, 7}^+$\endgroup};

\node(a) at (1.92, .73) {\begingroup\makeatletter\def\f@size{10}\check@mathfonts$e_{2, 8}^-$\endgroup};
\node(a) at (2.9, .73) {\begingroup\makeatletter\def\f@size{10}\check@mathfonts$e_{2, 8}^+$\endgroup};

\node(a) at (2.8, -.2) {\begingroup\makeatletter\def\f@size{10}\check@mathfonts$e_{2, 9}^-$\endgroup};
\node(a) at (3.1, .3655) {\begingroup\makeatletter\def\f@size{10}\check@mathfonts$e_{2, 9}^+$\endgroup};

\node[circle,fill=black,inner sep=0pt,minimum size=3.5pt] (a) at (2.0859, .3655) {};
\node[circle,fill=black,inner sep=0pt,minimum size=3.5pt] (a) at  (2.2969, 0) {};

\node[circle,fill=black,inner sep=0pt,minimum size=3.5pt] (a) at  (2.2265, .6089) {};
\node[circle,fill=black,inner sep=0pt,minimum size=3.5pt] (a) at  (2.6485, .6089) {};

\node[circle,fill=black,inner sep=0pt,minimum size=3.5pt] (a) at (2.578, 0) {};
\node[circle,fill=black,inner sep=0pt,minimum size=3.5pt] (a) at  (2.789, .3655) {};


%
%

\draw[ - ] (0,0) to (3/2, 2.598); 
\draw[ - ] (3/2, 2.598) to (3,0); 
\draw[ - ] (0, 0) to (3, 0); 
\draw[ - ] (0.5625, .975) to (1.125, 0); 
\draw[ - ] (2.4375, .975) to (1.875,0); 
\draw[ - ] (0.9375, 1.6234) to (2.0625, 1.6234); 
\draw[ - ] (0.211, .3655) to (0.421875, 0); 
\draw[ - ] (0.35155, .609) to (0.7735, .609); 
\draw[ - ] (0.703, 0) to (0.914, .3655); 

\draw[ - ] (1.1485, 1.9895) to (1.3595, 1.6238); 
\draw[ - ] (1.289, 2.2325) to (1.7109, 2.2325); 
\draw[ - ] (1.641, 1.6238) to (1.8515, 1.9895); 

\draw[ - ] (2.0859, .3655) to (2.2969, 0); 
\draw[ - ] (2.2265, .6089) to (2.6485, .6089); 
\draw[ - ] (2.789, .3655) to (2.578, 0); 
\draw[ - ] (0.0791, .1370) to (0.158, 0); 
\draw[ - ] (0.1318, .2283) to (0.290, .2283); 
\draw[ - ] (0.264, 0) to (0.343, .137); 

\draw[ - ] (1.368, 2.37) to (1.446,2.232); 
\draw[ - ] (1.421, 2.46) to (1.579,2.46); 
\draw[ - ] (1.553, 2.233) to (1.632, 2.37); 

\draw[ - ] (2.657, 0.137) to (2.738,0); 
\draw[ - ] (2.710, .228) to (2.868,.228); 
\draw[ - ] (2.921, .137) to (2.842, 0); 

\draw[ - ] (1.954, 0.137) to (2.033,0); 
\draw[ - ] (2.007, .228) to (2.165,.228); 
\draw[ - ] (2.139, 0) to (2.212, .137); 

\draw[ - ] (.431, 0.746) to (.509,0.609); 
\draw[ - ] (.4834, .837) to (.642,.837); 
\draw[ - ] (.615, 0.609) to (.6945, .746); 

\draw[ - ] (.782, .137) to (.863, 0); 
\draw[ - ] (.835, .228) to (.993, .228); 
\draw[ - ] (1.046, .137) to (.9668, 0); 

\draw[ - ] (1.017, 1.761) to (1.096, 1.624); 
\draw[ - ] (1.07, 1.852) to (1.223, 1.852); 
\draw[ - ] (1.201, 1.624) to (1.28, 1.761); 

\draw[ - ] (2.306, .746) to (2.385, .609); 
\draw[ - ] (2.358, .837) to (2.517, .837); 
\draw[ - ] (2.490, .609) to (2.570, .746); 

\draw[ - ] (1.72, 1.761) to (1.8, 1.624); 
\draw[ - ] (1.773, 1.852) to (1.931, 1.852); 
\draw[ - ] (1.983, 1.761) to (1.904, 1.624); 
\end{tikzpicture} 
\caption{Example of edges $e_{n_0, h}^\pm$ for $n_0 = 2$ and $1\leq h\leq 9$.}
\label{ssgedges}
\end{center}
\end{figure}
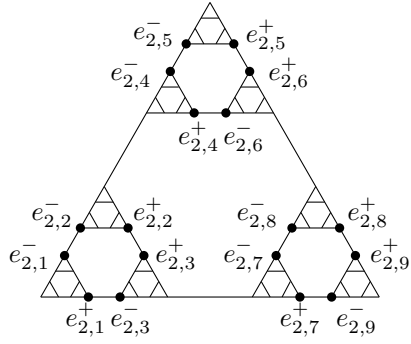

\begin{proposition}\label{psi}
Let $\calH^{d_\al}$ be the $d_\al$-dimensional Hausdorff probability measure on $K_\al$ and $\psi_\al: C(K_\al) \to \R$ given by
$$\psi_\al(f) = \int_{K_\al} f(x) \ d\calH^{d_\al}.$$
Then the sequence $\{\psi_{\al, n}\}$ converges to $\psi_\al$ in the weak-$^\ast$topology on the dual space of $C(K_\al)$.
\end{proposition}

\begin{proof} 
Let $\varepsilon >0$. Since $f$ is uniformly continuous on $K_\al$, there exists an $n_0 \in\N$ such that for all $h\in \{1, \dots, 3^{n_0}\}$ and any two points $x, y$ inside or on the triangle $\Delta_{n_0, h}$, we have $|f(x)-f(y)|< \varep$. Let $n>n_0$ and define $u_{n_0, h}^n: C(K_\al) \to \R$ by
$$u_{n_0, h}^n (f) = \frac{1}{2\cdot3^{(n-n_0)}} \sum_{\ep\in \Sn} \sum_{s \in \{+, -\}} f(\ep^s).$$ 
Then 
\begin{align*}
& \left |u_{n_0, h}^n(f) - \frac{f(e_{n_0, h}^-)+f(e_{n_0, h}^+)}{2} \right|\\
& = \left| \frac{1}{2\cdot3^{(n-n_0)}} \sum_{\ep\in\mathcal{S}^n_{n_0, h}} \sum_{s \in \{+, -\}} f(\ep^s) -  \frac{1}{2}\sum_{s \in \{+, -\}} f(e_{n_0,h}^s)\right|\\
&=  \left| \frac{1}{2\cdot 3^{(n-n_0)}} \sum_{\ep\in\mathcal{S}^n_{n_0, h}} \sum_{s \in \{+, -\}} f(\ep^s) - \frac{1}{2\cdot 3^{(n-n_0)}} \sum_{\ep\in\mathcal{S}^n_{n_0, h}}\sum_{s \in \{+, -\}} f(e_{n_0,h}^s)\right|\\
&\leq \frac{1}{2\cdot 3^{(n-n_0)}} \sum_{\ep\in\mathcal{S}^n_{n_0, h}} \sum_{s \in \{+, -\}}\left | f(\ep^s) - f(e_{n_0,h}^s)\right|\\
&< \varep.
\end{align*}
Notice that $\psi_{\al, n}(f) =  3^{-n_0} \sum_{h=1}^{3^{n_0}} u_{n_0, h}^n (f)$; using the above estimate,
\begin{align*}
\left|\psi_{\al, n}(f) -\psi_{\al, n_0}(f)\right|
&= \left |3^{-n_0} \sum_{h=1}^{3^{n_0}} \left(u_{n_0, h}^n(f)- \frac{1}{2} \sum_{s\in\{+, -\}} f(e_{n_0, h}^s)\right)  \right|\\
& \leq3^{-n_0} \sum_{h=1}^{3^{n_0}}  \left |u_{n_0, h}^n(f)- \frac{1}{2} \sum_{s\in\{+, -\}} f(e_{n_0, h}^s)\right|\\
&<\varep.
\end{align*}
Thus the sequence $\{\psi_{\al, n}\}_{n\geq1}$ converges to a functional $\psi_\al$ in the weak-$^\ast$topology on the dual of $C(K_\al)$. We now show that $\psi_\al$ is self-affine and hence must induce the $d_\al$-dimensional Hausdorff measure, which is the unique measure on $SG$ with the self-affinity (really, self-similarity) property. Consider the desired equality:
\begin{equation}\label{sspsi}
\psi_\al(f) = \frac{1}{3} \sum_{j=1}^3 \psi_\al(f\circ F_j) \ \ \ \ \ \ \ \ \ \ \text{for } f\in C(K_\al).
\end{equation}
Indeed, notice that
\begin{align*}
\frac{1}{3} \sum_{j=1}^3 \psi_{\al, n}(f\circ F_j) 
&= \frac{1}{3} \sum_{j=1}^3 3^{-n}\cdot 2^{-1}\sum_{\ep\in \calJ_n\setminus \calJ_{n-1}} \sum_{s \in \{+, -\}} f(F_j(\ep^s)) \\
&= \frac{1}{3^{n+1}\cdot 2} \sum_{\ep\in \calJ_{n+1}\setminus \calJ_{n}} \sum_{s \in \{+, -\}} f(\ep^s) \\
&= \psi_{\al, n+1}(f)
\end{align*}
and letting $n\to \infty$ shows that (\ref{sspsi}) holds. Thus, $\psi_{\al, n} \to \psi_\al$, where $\psi_\al(f) = \int_{K_\al} f(x) \ d\calH^{d_\al}.$
\end{proof} 

%
%
\begin{lemma}\label{dixconst}
For the spectral triple $S(K_\al)$ of dimension $\ds=\ds(K_\al)=d_\al$,
$$Tr_w(|D_\al|^{-\ds}) = \dfrac{2^{\ds+1}(2^{\ds}-1)\zeta(\ds)(3+3\al^\ds)}{\ds\cdot\pi^\ds(2^\ds\log(2)-3(1-\al)^\ds\log(1-\al))}.$$
\end{lemma}
\begin{proof} 
Using Theorem \ref{co} and Lemma \ref{newtr}, as well as the fact that $\ds >1$,
\begin{align*}
Tr_w(|D_\al|^{-\ds}) &= \lim_{s\to 1+} (s-1)\text{tr} (|D_\al|^{-\ds s})\\
&= \lim_{s\to 1+} (s-1) \dfrac{2^{\ds s+1}(1-2^{-{\ds s}})\zeta({\ds s})}{\pi^{\ds s}}\dfrac{2^{\ds s}(3+3\al^{\ds s})}{2^{\ds s} - 3(1-\al)^{\ds s}}\\
&=\dfrac{ 2^{\ds+1}(2^{\ds}-1)\zeta({\ds})(3+3\al^{\ds})}{\pi^{\ds}} \lim_{s\to 1+} (s-1) \dfrac{1}{2^{\ds s}- 3(1-\al)^{\ds s}}\\
&=\dfrac{ 2^{\ds+1}(2^{\ds}-1)\zeta({\ds})(3+3\al^{\ds})}{\pi^{\ds}} \lim_{s\to 1+}\  \dfrac{1}{\ds2^{\ds s}\log(2)- \ds3(1-\al)^{\ds s}\log(1-\al)}\\
&=\dfrac{ 2^{\ds+1}(2^{\ds}-1)\zeta({\ds})(3+3\al^{\ds})}{\ds\cdot \pi^{\ds}(2^{\ds}\log(2)- 3(1-\al)^{\ds}\log(1-\al))}.
\end{align*}
\end{proof} 

The spectral dimension $\ds = \ds(K_\al)$ is the same as the Hausdorff dimension $d_\al$. In what follows we will write $\ds$ in order to showcase how our operator algebraic tools recover fractal geometric data like the Hausdorff measure on $K_\al$. 

\begin{theorem}
The spectral triple $S(K_\al)$ recovers the $\ds$-dimensional Hausdorff measure, $\calH^\ds$, on $K_\al$ via the formula
$$ Tr_w(\pi_\al(f)|D_\al|^{-\ds}) = c_\ds \int_{K_\al} f \ d\calH^\ds$$
for all $f\in C(K_\al)$. Moreover,
$$c_\ds = \dfrac{2^{\ds+1}(2^{\ds}-1)\zeta(\ds)(3+3\al^\ds)}{\ds\cdot\pi^\ds(2^\ds\log(2)-3(1-\al)^\ds\log(1-\al))}.$$
\end{theorem}
\begin{proof} 
Let $\varep >0$ and $f\in C(K_\al)$. Let $n_0\in\N$ such that for all $h\in \{1, 2,\dots, 3^{n_0}\}$ and any two points $x, y$ inside or on the triangle $\Delta_{n_0, h}$, we have $|f(x)-f(y)|< \varep$. 
Choose $n>n_0$ and define $u_{n_0, h}^n: C(K_\al) \to \R$, as in the proof of Proposition \ref{psi}:
$$u_{n_0, h}^n (f) = \frac{1}{2\cdot3^{(n-n_0)}} \sum_{\ep\in \Sn} \sum_{s \in \{+, -\}} f(\ep^s).$$ 
Denote by $K_{n_0, h}$ the portion of $K_\al$ contained in the triangle $\Delta_{n_0, h}$. 
Let $I_{n_0, h}= 1$ in $C(K_{n_0, h})$, $f_{n_0, h} = f|_{K_{n_0, h}}$ in $C(K_{n_0, h})$, and $f_{J_{n_0}} = f|_{J_{n_0}}$ in $C(\overline{J}_{n_0})$.
Then
\begin{align*}
\left|u_{n_0, h}^n(f)I_{n_0, h}(x)-f_{n_0, h}(x)\right|
=&\left | \frac{1}{2\cdot 3^{(n-n_0)}} \sum_{\ep \in \Sn} \sum_{s \in \{+, -\}} f(\ep^s)I_{n_0, h}(x) - f_{n_0, h}(x)\right|\\
=&\left |  \frac{1}{2\cdot 3^{(n-n_0)}} \sum_{\ep \in \Sn} \sum_{s \in \{+, -\}} (f(\ep^s)I_{n_0, h}(x) - f_{n_0, h}(x))\right|\\
\leq & \  \frac{1}{2\cdot 3^{(n-n_0)}} \sum_{\ep \in \Sn} \sum_{s \in \{+, -\}}\left| f(\ep^s)I_{n_0, h}(x) - f_{n_0, h}(x)\right|\\
< &\ \varep ;
\end{align*}
 so
 \begin{equation}\label{ineq}
 (u_{n_0, h}^n(f)-\varep) I_{n_0, h} < f_{n_0,h} < (u_{n_0, h}^n(f)+\varep) I_{n_0, h}. 
 \end{equation}

Next note that one can define a spectral triple for $K_{n_0, h}$ and one for $\overline{J}_{n_0}$ by deleting the summands from the spectral triple for $K_\al$ which correspond to edges outside of $K_{n_0, h}$ or outside of $\overline{J}_{n_0}$, respectively. The argument that this construction does indeed gives a spectral triple for $K_{n_0, h}$ is the same as that for the spectral triple for $K_\al$. In the case of $\overline{J}_{n_0}$, that this deletion of summands still gives a spectral triple follows from Proposition 5.1 in \cite{cil}. Denote by $\trw( \pi_{n_0, h}(f_{n_0, h})|D_{n_0, h}|^{-\ds})$ and $\trw(\pi_{J_{n_0}}(f_{J_{n_0}})|D_{J_{n_0}}|^{-\ds})$ the positive linear functionals respectively associated to these triples. Using the fact that
$$\sigma(D_\al) = \bigcup_{h=1}^{3^{n_0}} \sigma(D_{n_0, h}) \cup \sigma(D_{J_{n_0}})$$ 
and that as operators $\pi_\al(f) = \oplus_{h=1}^{3^{n_0}} \pi_{n_0, h}(f_{n_0, h}) \oplus \pi_{J_{n_0}}(f_{J_{n_0}}),$ we have
 $$\trw(\pi_\al(f)|D_\al|^{-\ds}) = \sum_{h=1}^{3^{n_0}} \trw( \pi_{n_0, h}(f_{n_0,h})|D_{n_0, h}|^{-\ds})  + \trw(\pi_{J_{n_0}}(f_{J_{n_0}})|D_{J_{n_0}}|^{-\ds}). $$
Next we show that $\trw(\pi_{J_{n_0}}(f_{J_{n_0}})|D_{J_{n_0}}|^{-\ds})=0$. 
Note that 
 	$$\lim_{s\to 1^+} (s-1) \text{tr} (|D_{J_{n_0}}|^{-\ds s})
 	= \lim_{s\to 1^+} (s-1) \sum_{j=1}^{n_0} \beta_{\ds s} \ 3^{j} \al^{\ds s} \left( \frac{1-\al}{2}\right)^{\ds s(j-1)}\\
 	 = 0$$
since 
$$\sum_{j=1}^{n_0} \beta_{\ds s} 3^{j} \al^{\ds s} \left( \frac{1-\al}{2}\right)^{\ds s(j-1)}$$ 
converges as $s\to 1^+$ and hence 
$$\trw(|D_{J_{n_0}}|^{-\ds}) = \lim_{s\to 1^+} (s-1) \text{tr} (|D_{J_{n_0}}|^{-\ds s})  = 0.$$ 
For a continuous function $f_{J_{n_0}}$ on the closed set $\overline{J}_{n_0}$, there is an $M$ such that $|f_{J_{n_0}}| \leq M$. Since $\trw(\pi_{J_{n_0}}(\cdot)|D_{J_{n_0}}|^{-\ds})$ is a positive linear functional on $C(\overline{J}_{n_0})$, we know that
$$\trw(\pi_{J_{n_0}}(f_{J_{n_0}})|D_{J_{n_0}}|^{-\ds})\leq \trw(M|D_{J_{n_0}}|^{-\ds}) = M\trw(|D_{J_{n_0}}|^{-\ds}).$$
It follows that $\trw(\pi_{J_{n_0}}(f_{J_{n_0}})|D_{J_{n_0}}|^{-\ds})=0$ and 
$$\trw(\pi_\al(f)|D_\al|^{-\ds}) = \sum_{h=1}^{3^{n_0}} \trw( \pi_{n_0, h}(f_{n_0,h})|D_{n_0, h}|^{-\ds}).$$
Also note that 
$$\trw( \pi_{n_0, h}(I_{n_0,h})|D_{n_0, h}|^{-\ds})= 3^{-n_0}  \trw( \pi_\al(I)|D_\al|^{-\ds})= 3^{-n_0}  \trw(|D_\al|^{-\ds}).$$
Using the inequalities (\ref{ineq}), the fact that $\trw(\pi_\al(\cdot)|D_\al|^{-\ds})$ is a positive linear functional on $C(K_\al)$, and summing, gives
\begin{equation*}\label{trineq}
\sum_{h=1}^{3^{n_0}}(u_{n_0, h}^n(f)-\varep) (3^{-n_0}\trw(|D_\al|^{-\ds})) \leq \trw( \pi_\al(f)|D_\al|^{-\ds}) \leq \sum_{h=1}^{3^{n_0}} (u_{n_0, h}^n(f)+\varep) (3^{-n_0}\trw(|D_\al|^{-\ds})),
\end{equation*}
which is the same as
\begin{equation}\label{trineq2}
\left| \trw( \pi_\al(f)|D_\al|^{-\ds})-3^{-n_0}\trw(|D_\al|^{-\ds})\sum_{h=1}^{3^{n_0}}u_{n_0, h}^n(f)\right| \leq 3^{-n_0}\trw(|D_\al|^{-\ds})\varep.
\end{equation}
In Proposition \ref{psi} we showed that the functionals $\psi_{\al, n}$ converge to the functional $\psi_\al$ in the weak-$^\ast$ topology on the dual of $C(K_\al)$ and that for $n\geq n_0$ we can write $\psi_{\al, n}(f) = 3^{-n_0} \sum_{h=1}^{3^{n_0}} u_{n_0, h}^n (f)$; so (after possibly choosing $n_0$ to be larger)
$$\left|  \psi_\al(f) - 3^{-n_0} \sum_{h=1}^{3^{n_0}} u_{n_0, h}^n (f)\right| < \varep$$
and multiplying by $\trw( |D_\al|^{-\ds})$, 
\begin{equation}\label{trineq3}
\left|  \trw( |D_\al|^{-\ds})\psi_\al(f) - \trw(|D_\al|^{-\ds})3^{-n_0} \sum_{h=1}^{3^{n_0}} u_{n_0, h}^n (f)\right| \leq \trw(|D_\al|^{-\ds})\varep.
\end{equation}
Note by Lemma \ref{dixconst} that the value $\trw( |D_\al|^{-\ds})$ is positive so the inequality above is preserved. 
Then using the estimates (\ref{trineq2}) and (\ref{trineq3}),
\begin{align*}
&\left|\trw( \pi_\al(f)|D_\al|^{-\ds})- \trw(|D_\al|^{-\ds})\psi_\al(f)\right|\\
&\leq \left| \trw( \pi_\al(f)|D_\al|^{-\ds})-\trw(|D_\al|^{-\ds})\psi_{\al, n}(f)\right| +\left|\trw(|D_\al|^{-\ds})\psi_{\al, n}(f) -\trw(|D_\al|^{-\ds})\psi_\al(f)\right|\\
&\leq 3^{-n_0}\trw(|D_\al|^{-\ds})\varep + \trw(|D_\al|^{-\ds}) \varep\\
& = (3^{-n_0}+1) \ \trw(|D_\al|^{-\ds})\varep,
\end{align*}
from which it follows that 
$$\trw(\pi_\al(f)|D_\al|^{-\ds})= \trw(|D_\al|^{-\ds}) \psi_\al(f).$$
Using Lemma \ref{dixconst} we have
$$ \trw(\pi_\al(f)|D_\al|^{-\ds}) = c_\ds \int_{K_\al} f \ d\calH^{\ds},$$
as desired. 
\end{proof} 

\section{Conclusion}
The results in this paper are intended to further develop the intersection between fractal geometry, noncommutative geometry, and analysis on fractals. We have proved that the conjecture in \cite{lapsar} regarding the recovery of the Hausdorff measure with respect to the geodesic distance on the harmonic gasket by the Dixmier trace is false. 
The Dixmier trace on the harmonic gasket can recover the standard self-affine measure on the harmonic gasket, but this measure is not the same as the Hausdorff measure with respect to the geodesic distance. 
We have also shown that a spectral triple built on the edges of the stretched Sierpinski gasket can be used to recover the Hausdorff dimension, the geodesic metric, and the Hausdorff measure. This is especially interesting since the stretched Sierpinski gasket is a self-affine space and not a self-similar space. 

In the future we will consider the question of constructing a spectral triple on the harmonic gasket which will recover the Hausdorff measure. There are results concerning the asymptotics of the Laplacian on the harmonic gasket with respect to the Hausdorff measure (see \cite{kaj}) and it may prove useful to find a connection between this fractal analysis and the operator algebraic tools that come with spectral triples. Also, there are some results on the use of spectral triples to recover energy forms on fractal sets like the Sierpinski gasket; see \cite{cgis}. One can construct an energy form on the stretched Sierpinski gasket and it would be interesting to assemble a spectral triple that recovers the energy on the stretched Sierpinski gasket. 



\frenchspacing
\begin{thebibliography}{9}
\bibitem{ar}
	P. Alonso Ruiz, 
	Dirichlet forms on non self-similar sets: Hanoi attractors and the Sierpinski gasket,
	Ph.D. thesis, University of Siegen, Germany, 2013. 

\bibitem{arf}
	P. Alonso Ruiz and U. Freiberg, 
	Weyl asymptotics for Hanoi attractors,
	\emph{Forum Mathematicum},
	De Gruyter, 2013,
	arXiv:1307.6719v4.
	
\bibitem{afk}
	P. Alonso Ruiz, U. Freiberg, and J. Kigami,
	Completely symmetric resistance forms on the stretched Sierpinski gasket,
	2016,
	arXiv:1606.08582.

\bibitem{artep}
	P. Alonso Ruiz, J. D. Kelleher, and A. Teplyaev,
	Energy and Laplacian on Hanoi-Type fractal quantum graphs,
	2015,
	arXiv: 1408.4658.
	
 \bibitem{cgis}
	F. Cipriani, D. Guido, T. Isola, and J-L. Sauvageot,
	Spectral triples for the Sierpinski gasket
	\emph{J. Funct. Anal.} \textbf{266} (2014) 4809--4869,
	arXiv:1112.6401.

\bibitem{ci}
	E. Christensen and C. Ivan,
	Spectral triples for AF C $^\ast$-algebras and metrics on the Cantor set,
	\emph{J. Operator Theory} \textbf{56} (2006), 17--46.	
	
 \bibitem{cil}
	E. Christensen, C. Ivan, and M. L. Lapidus,  
	Dirac operators and spectral triples for some fractal sets built on curves,
	\emph{Advances in Mathematics} (1) \textbf{217} (2008), 42--78.

\bibitem{co}
	A. Connes,
	\emph{Noncommutative Geometry,}
	Academic Press, San Diego, 1994.
	
\bibitem{co2}
	A. Connes,
	Compact metric spaces, Fredholm modules, and hyperfiniteness,
	\emph{Ergodic Theory and Dynamical Systems} $\mathbf{9}$ (1989), 207--220. 
	
\bibitem{ed}
	R. E. Edwards,
	\emph{Functional Analysis: Theory and Applications,}
	Dover Publications Inc., New York, 1995.

\bibitem{fa}
	K. Falconer,
	\emph{Fractal Geometry: Mathematical Foundations and Applications,}
	John Wiley \& Sons Ltd., 1990.
	
\bibitem{gu}
	D. Guido and T. Isola,
	Dimensions and singular traces for spectral triples, with applications to fractals,
	\emph{J. Funct. Anal.} \textbf{203} (2003), 362--400.

\bibitem{gu2}
	D. Guido and T. Isola,
	Dimensions and spectral triples for fractals in $\R^N$,
	in Advances in \emph{Operator Algebras and Mathematical Physics} (F. Boca et al., eds.) 
	Theta Ser. Adv. in Mathematics, \textbf{5}, Theta, Bucharest, 2005, 89--108.
	
\bibitem{hut}
	J. E. Hutchinson,
	Fractals and self similarity,
	\emph{Indiana University Mathematics Journal} \textbf{30} (1981), 713--747.

\bibitem{kaj}
	N. Kajino,
	Analysis and geometry of the measurable Riemannian structure on the Sierpinski gasket, 
	 in \emph{Fractal Geometry and Dynamical Systems in Pure and Applied Mathematics I: Fractals in Pure Mathematics} 
	 (D. Carfi, M.L. Lapidus, E.P.J. Pearse and M. van Frankenhuijsen, eds.), 
	 Contemporary Math., \textbf{600}, Amer. Math. Soc., Providence, RI (2013), 91--133.

\bibitem{kaj1}
	N. Kajino,
	In e-mail corresondence between Dr. Kajino and the author. 
	September 29, 2016. 

\bibitem{kig}
	J. Kigami,
	\emph{Analysis on Fractals},
	Cambridge Tracts in Math. vol. 143, 
	Cambridge University Press, 2001.
	
\bibitem{kig1}
	J. Kigami,
	Harmonic metric and Dirichlet form on the Sierpinski gasket,
	in \emph{Asymptotic Problems in Probability Theory: Stochastic Models and Diffusion on Fractals} (Sanda/Kyoto, 1990), (K.D. Elworthy and N. Ikeda, eds.),
	Pitman Res. Math., \textbf{283} (1993), 210--218.
	
\bibitem{kig3}
	J. Kigami,
	Measurable Riemannian geometry on the Sierpinski gasket: the Kusuoka measure and the Gaussian heat kernel estimate, 
	\emph{Mathematische Annalen.} (4) \textbf{340} (2008), 781--804.
	
\bibitem{lap2}
	M. L. Lapidus,
	Analysis on fractals, Laplacians of self-similar sets, noncommutative geometry, and spectral dimensions,
	\emph{Topological Methods in Nonlinear Analysis} \textbf{4} (1997), 135--195. 

\bibitem{lap1}
	M. L. Lapidus,
	Towards a noncommutative fractal geometry? Laplacians and volume measures on fractals,
	in \emph{Harmonic Analysys and Nonlinear Differential Equations}, Contemp. Math., \textbf{208}, 
	Amer. Math. Soc., Providence, RI, 1997, 211--252. 
	
\bibitem{lappo}
	M. L. Lapidus and C. Pomerance,
	The Riemann zeta function and the one-dimensional Weyl-Berry conjecture for fractal drums,
	\emph{Proc. London Math. Soc.} (3) \textbf{66} (1993) 41--69. 
	
\bibitem{lapsar}
	M. L. Lapidus and J. Sarhad, 
	Dirac operators and geodesic metric on the harmonic Sierpinski gasket and other fractal sets,
	\emph{J. of Noncom. Geo.} \textbf{8} (2014), 947--985.

\bibitem{lord}
	S. Lord, F. Sukochev, D. Zanin, 
	\emph{Singular Traces: Theory and Applications},
	vol. 46. De Gruyter Studies in Mathematics,
	2013.
	
\bibitem{rif}
	M. A. Rieffel,
	Metrics on states from actions of compact groups,
	\emph{Doc. Math.} 3 (1998) 215--229

\bibitem{rif2}
	M. A. Rieffel,
	Metrics on state spaces,
	\emph{Doc. Math.}  4 (1999) 559--600. arXiv:math/9906151

\bibitem{rif3}
	M. A. Rieffel,
	Compact quantum metric spaces,
	in: \emph{Operator Algebras, Quantization, and Non-Commutative Geometry}, 
	Contemp. Math., \textbf{365}, Amer. Math. Soc., 2004, 315--330.
	
\bibitem{stri}
	R. S. Strichartz,
	\emph{Differential Equations on Fractals: A Tutorial},
	Princeton University Press, 2006.
\end{thebibliography}
\end{document}